\documentclass[10pt]{amsart}
\usepackage{amssymb}
\usepackage{bm}
\usepackage{graphicx}
\usepackage[centertags]{amsmath}
\usepackage{amsfonts}
\usepackage{amsthm}
\usepackage{amsbsy}
\usepackage{mathtools}
\usepackage{mathrsfs}
\usepackage{cases}
\usepackage{xfrac}
\usepackage[all]{xy}
\usepackage[linktocpage=true]{hyperref}
\usepackage{hyperref}
\hypersetup{
    colorlinks=true,
    linkcolor=blue,
    filecolor=blue,
    citecolor=blue,
    urlcolor=cyan}
\usepackage[margin=4cm]{geometry}
\newtheorem{thm}{Theorem}[section]

\newtheorem{lem}[thm]{Lemma}

\newtheorem{prop}[thm]{Proposition}

\newtheorem{fact}[thm]{Fact}

\newtheorem{problem}[thm]{Problem}
\newtheorem{defn}[thm]{Definition}
\theoremstyle{remark}
\newtheorem{rem}[thm]{Remark}
\newtheorem{examp}[thm]{Example}

\newcommand{\nn}{\mathbb{N}}
\newcommand{\ee}{\varepsilon}
\newcommand{\con}{\smallfrown}
\newcommand{\meg}{\geqslant}
\newcommand{\mik}{\leqslant}

\begin{document}

\title[A structure theorem for stochastic processes]{A structure theorem for stochastic processes indexed by the discrete hypercube}

\author{Pandelis Dodos and Konstantinos Tyros}

\address{Department of Mathematics, University of Athens, Panepistimiopolis 157 84, Athens, Greece}
\email{pdodos@math.uoa.gr}

\address{Department of Mathematics, University of Athens, Panepistimiopolis 157 84, Athens, Greece}
\email{ktyros@math.uoa.gr}


%------------------------Abstract-------------------------------%

\begin{abstract}
Let $A$ be a finite set with $|A|\geqslant 2$, let $n$ be a positive integer, and let $A^n$ denote the discrete
$n$-dimensional hypercube (that is, $A^n$ is the Cartesian product of $n$ many copies of $A$). Given a family
$\langle D_t:t\in A^n\rangle$ of measurable events in a probability space (a stochastic process), what structural
information can be obtained assuming that the events $\langle D_t:t\in A^n\rangle$ are not behaving as if they were
independent? We obtain an answer to this problem (in a strong quantitative sense) subject to a mild ``stationarity"
condition. Our result has a number of combinatorial consequences, including a new (and the most informative so far)
proof of the density Hales--Jewett theorem.
\end{abstract}

\maketitle

\tableofcontents

%--------------------------Introduction-----------------------%

\section{Introduction} \label{s1}

\numberwithin{equation}{section}

\subsection{Motivation/Overview} \label{ss1.1}

Let $I$ be a nonempty finite index set, let $\langle E_i:i\in I\rangle$ and $\langle D_i : i\in I\rangle$
be stochastic processes (families of measurable events) in a probability space
$(\Omega,\mathcal{F},\mathbb{P})$ with $\mathbb{P}(E_i)=\mathbb{P}(D_i)=\ee>0$ for all $i\in I$, and assume that
the events $\langle E_i:i\in I\rangle$ are independent. We wish to compare the distributions of the random variables
\[ \textsf{X}=\sum_{i\in I} \mathbf{1}_{E_i} \ \text{ and } \  \textsf{Y}=\sum_{i\in I} \mathbf{1}_{D_i} \]
with the main goal being here that of transferring information from the distribution of~$\textsf{X}$ (which we understand very well)
to the distribution of $\textsf{Y}$, an object on which we have a priori no control.  A classical method for doing so
is by comparing the moments of $\textsf{X}$ and $\textsf{Y}$ (see, \emph{e.g.}, \cite{Du}), a task which essentially
reduces\footnote{This is because $\textsf{X}$ and $\textsf{Y}$ are both sums of indicator functions.}
to that of comparing the joint probability of $\langle D_i:i\in F\rangle $ with the expected value $\ee^{|F|}$ as $F$ varies
over all nonempty subsets of the index set $I$. Thus, assuming that the random variables $\textsf{X}$ and $\textsf{Y}$ are
\textit{not} close in distribution, then one is led to the following problem.
\begin{problem} \label{prob1}
Let $F\subseteq I$ be nonempty, let $\sigma>0$, and assume that
\[ \Big| \mathbb{P}\Big( \bigcap_{i\in F} D_i\Big) -\ee^{|F|}\Big| \meg \sigma. \]
What structural information can be obtained for the process $\langle D_i : i\in I\rangle$?
\end{problem}

\subsubsection{The combinatorial content} \label{sss1.1.1}

We will study Problem \ref{prob1} in the case where the index set $I$ is a \textit{discrete hypercube}, that is, a set of the form
\begin{equation} \label{e1.1}
A^n \coloneqq \underbrace{A\times\cdots\times A}_{n-\mathrm{times}}
\end{equation}
where $A$ is a finite set with $|A|\meg 2$ and $n$ is a positive integer which is commonly referred to
as the \textit{dimension} of the hypercube $A^n$. This choice of the index set is by no means arbitrary
and it is ultimately related to the density Hales--Jewett theorem, a deep result due to Furstenberg and
Katznelson \cite{FK2} with numerous consequences in combinatorics, number theory, and theoretical computer science.

In order to properly discuss this relation we need to recall some basic definitions.
Let $A$ and $n$ be as above, and fix a letter $x\notin A$ which we view
as a variable. A~\textit{variable word over $A$ of length $n$} is a finite sequence of length $n$ having values
in $A\cup \{x\}$ where the letter $x$ appears at least once. If $v$ is a variable word over $A$ of length $n$
and $\alpha\in A$, then let $v(\alpha)$ denote the unique element of $A^n$ which is obtained by replacing
every appearance of the letter $x$ in $v$ with $\alpha$. (For instance, if $A=\{\alpha,\beta,\gamma\}$ and
$v=(\alpha,x,\gamma,\beta,x)$, then $v(\beta)=(\alpha,\beta,\gamma,\beta,\beta)$.) A \emph{combinatorial line}
of $A^n$ is a set of the form $\{v(\alpha):\alpha\in A\}$ where $v$ is a variable word over $A$ of
length $n$ (see \cite{GRS,HJ}).

We are now in a position to recall the density Hales--Jewett theorem. We will state a probabilistic
version---see, \emph{e.g.}, \cite[Proposition 2.1]{FK2}---which is closer in spirit to our discussion. The relation between
this probabilistic version and the more well-known combinatorial form which refers to dense subsets of discrete
hypercubes will be discussed in Section \ref{s4}.
\begin{thm} \label{PHJ1.1}
For every integer $k\meg 2$ and every $0<\ee\mik 1$ there exists a positive integer $\mathrm{PHJ}(k,\ee)$ with
the following property. Let $A$ be a set with $|A|=k$, let $n\meg \mathrm{PHJ}(k,\varepsilon)$ be an integer,
and let $\langle D_t:t\in A^n\rangle$ be a stochastic process in a probability space $(\Omega,\mathcal{F},\mathbb{P})$
such that\, $\mathbb{P}(D_t)\meg \varepsilon$ for every $t\in A^n$. Then there exists a combinatorial line $L$ of $A^n$ such that
\[ \mathbb{P}\Big(\bigcap_{t\in L} D_t\Big)> 0. \]
\end{thm}
Of course, Theorem \ref{PHJ1.1} is straightforward if the events $\langle D_t:t\in A^n\rangle$ are independent.
Thus, the core of the theorem is to understand what happens when the events are not behaving as if they were independent,
which is clearly an instance of Problem \ref{prob1}.

\subsubsection{Deviating from the expected value: examples} \label{sss1.1.2}

To gain insight on the kind of structure one expects to obtain in Problem \ref{prob1}, it is useful to give examples
of stochastic processes which exhibit non-independent behavior. Here and in the rest of this introduction, we will
restrict our discussion to correlations over combinatorial lines. This is mainly because of the combinatorial importance
of this case, but also because it is already quite representative of the behavior of correlations over more complicated sets.
\begin{examp} \label{examp-intro}
For concreteness we will work with the set $\{1,2,3\}$, but the argument can also be applied for any finite set $A$ with $|A|\meg 2$.
Let $n$ be an arbitrary positive integer. We start with a family $\langle E_s:s\in \{1,2\}^n\rangle$ of independent
events in a probability space $(\Omega,\mathcal{F},\mathbb{P})$ with equal probability $\ee>0$. Given $t\in\{1,2,3\}^n$ there are
two natural ways to ``project" it into $\{1,2\}^n$. Specifically, let $t^{3\to 1}$ and $t^{3\to 2}$ denote the unique
elements of $\{1,2\}^n$ which are obtained by replacing every appearance of $3$ in $t$ with $1$ and $2$ respectively.
(\emph{E.g.}, if $t=(3,2,1,3,1)\in \{1,2,3\}^5$, then $t^{3\to 1}=(1,2,1,1,1)$ and $t^{3\to 2}=(2,2,1,2,1)$.) Then let
$\langle D_t:t\in \{1,2,3\}^n\rangle$ be defined by setting $D_t\coloneqq E_{t^{3\to 1}}\cap E_{t^{3\to 2}}$ for every $t\in \{1,2,3\}^n$.
\end{examp}
Although the process $\langle D_t:t\in \{1,2,3\}^n\rangle$ in Example \ref{examp-intro} is, arguably, quite easy to define,
the analysis of its properties requires some work.
\medskip

\paragraph{1.1.2.1} \label{p1.1.2.1}
We first observe that for every $t\in\{1,2,3\}^n$ which contains $3$ we have
\begin{equation} \label{e1.2}
\mathbb{P}(D_t)=\ee^2.
\end{equation}
Since the density of the set of all elements of $\{1,2,3\}^n$ which do not contain $3$ decreases exponentially
with respect to the dimension $n$, we see that \eqref{e1.2} holds true for ``almost every" $t$.
\medskip

\paragraph{1.1.2.2} \label{p1.1.2.2}
The second basic property of the process $\langle D_t:t\in \{1,2,3\}^n\rangle$ concerns its correlations
over combinatorial lines. Specifically, let $L=\{v(1), v(2),v(3)\}$ be a combinatorial line of
$\{1,2,3\}^n$ where $v$ is a variable word over $\{1,2,3\}$ of length~$n$ which contains $3$. Then we have
\begin{equation} \label{e1.3}
\mathbb{P}\Big( \bigcap_{t\in L} D_t\Big)=\ee^4
\end{equation}
which implies that $\langle D_t:t\in\{1,2,3\}^n\rangle$ exhibits non-independent\footnote{Specifically, by \eqref{e1.2},
the expected probability in \eqref{e1.3} is $\ee^6$.} behavior.

However, identity \eqref{e1.3} shows yet another important property of this process.
More precisely, if $v_1, v_2$ are variable words over $\{1,2,3\}$ of length~$n$ which both contain $3$, then
\begin{equation} \label{e1.4}
\mathbb{P}\big( D_{v_1(1)}\cap D_{v_1(2)} \cap D_{v_1(3)} \big) = \mathbb{P}\big( D_{v_2(1)}\cap D_{v_2(2)} \cap D_{v_2(3)} \big).
\end{equation}
In other words, the correlations of $\langle D_t:t\in\{1,2,3\}^n\rangle$ over combinatorial lines are essentially constant.
This property is abstracted in the following definition which originates\footnote{The framework in \cite{FK2}
is somewhat different, but the essential content of Definition \ref{defn:stasionarity_lines} is present in that work.}
in the work of Furstenberg and Katznelson \cite{FK2}.
\begin{defn}[Stationarity] \label{defn:stasionarity_lines}
Let $A$ be a finite set with $|A|\meg 2$, let $n$ be a positive integer, let $\eta>0$, and let $\langle D_t: t\in A^n\rangle$
be a stochastic process in a probability space~$(\Omega,\mathcal{F},\mathbb{P})$. We say that $\langle D_t:t\in A^n\rangle$
is \emph{$\eta$-stationary (with respect to combinatorial lines)} if for every nonempty $\Gamma\subseteq A$
and every pair $v_1, v_2$ of variable words over $A$ of length $n$ we have
\begin{equation} \label{e1.5}
\Big|\mathbb{P}\Big( \bigcap_{\alpha\in \Gamma} D_{v_1(\alpha)}\Big) -
\mathbb{P}\Big( \bigcap_{\alpha\in\Gamma} D_{v_2(\alpha)}\Big)\Big| \mik \eta.
\end{equation}
$($In particular, if $\langle D_t:t\in A^n\rangle$ is an $\eta$-stationary process, then for every pair $L_1,L_2$ of combinatorial
lines of $A^n$ we have $\big|\mathbb{P}\big( \bigcap_{t\in L_1} D_t\big) - \mathbb{P}\big( \bigcap_{t\in L_2} D_t\big)\big|\!\mik\eta$.$)$
\end{defn}
Besides being very natural in this context\footnote{In particular note that, without assuming stationarity, one should instead study
the behavior of an appropriately weighted average, over all possible combinatorial lines, of the corresponding correlations of a process.},
stationarity is not a particularly restrictive condition. Indeed, it follows from a classical result due to Graham and Rothschild \cite{GR}
that stationary processes are the building blocks of arbitrary processes. (See Fact \ref{f3.1} in the main text.)
\medskip

\paragraph{1.1.2.3} \label{p1.1.2.3}
The last, and most significant, property of the process $\langle D_t:t\in\{1,2,3\}^n\rangle$ is its hidden
arithmetic structure which is described in the following definition.
\begin{defn}[Insensitivity] \label{insensitivity-intro}
Let $A$ be a finite set with $|A|\meg 2$, let $n$ be a positive integer, and let $\alpha, \beta\in A$ with $\alpha\neq\beta$.
\begin{enumerate}
\item[(1)] Let $s,t\in A^n$ and write $s=(s_1,\dots,s_n)$ and $t=(t_1,\dots, t_n)$. We say that $s,t$ are
\emph{$(\alpha,\beta)$-equivalent} if for every $i\in \{1,\dots,n\}$ and every $\gamma\in A\setminus \{\alpha,\beta\}$ we have that
$s_i=\gamma$ if and only if\, $t_i=\gamma$. $($Namely, $s,t$ are $(\alpha,\beta)$-equivalent if they possibly differ
only in the coordinates taking values in $\{\alpha,\beta\}$.$)$
\item[(2)] We say that a stochastic process $\langle D_t:t\in A^n\rangle$ in a probability space $(\Omega,\mathcal{F},\mathbb{P})$
is \emph{$(\alpha,\beta)$-insensitive} if\, $D_s=D_t$ for every $s,t\in A^n$ which are $(\alpha,\beta)$-equivalent.
\end{enumerate}
\end{defn}
The notion of insensitivity was introduced by Shelah \cite{Sh} in his proof of the Hales--Jewett theorem \cite{HJ},
though it was originally referring to subsets of discrete hypercubes and not to stochastic processes (see Definition
\ref{insensitivity-sets} in the main text); the difference, however, between the two frameworks is minor.
In the discrete setting, insensitivity is the analogue\footnote{This can be seen by identifying any nonempty
finite set $A$ with the interval $\{1,\dots,|A|\}$ and then projecting the hypercube $A^n$ into the integers
via the map $(\alpha_1,\dots,\alpha_n)\mapsto \sum_{i=1}^n \alpha_i\,|A|^{i-1}$.} of the concept of a
\textit{$($discrete$)$ Hilbert cube} which is ubiquitous in additive combinatorics and arithmetic Ramsey theory
(see, \emph{e.g.}, \cite{GRS,TV}). Insensitive processes have also been considered in \cite[Definition~5.1]{Au}.

Now, taking into account the definition of $t^{3\to 1}$ and $t^{3\to 2}$ in Example \ref{examp-intro}, it is easy to see that the
processes $\langle E_{t^{3\to 1}}:t\in \{1,2,3\}^n\rangle$ and $\langle E_{t^{3\to 2}}:t\in \{1,2,3\}^n\rangle$ are $(1,3)$- and
$(2,3)$-insensitive respectively. This property by itself yields that for \textit{every} variable word $v$ over $\{1,2,3\}$
of length~$n$ we have
\begin{equation} \label{e1.6}
D_{v(1)}\cap D_{v(2)} \cap D_{v(3)}= D_{v(1)}\cap D_{v(2)}.
\end{equation}
Note that identity \eqref{e1.6} implies, in a rather extreme way, that the events $D_{v(1)}, D_{v(2)}$ and $D_{v(3)}$
cannot be independent. Thus we have a structural explanation of the fact that $\langle D_t:t\in\{1,2,3\}^n\rangle$
exhibits non-independent behavior: it  is the intersection of insensitive processes.

\subsection{The main result} \label{ss1.2}

The following theorem (which is one of the main results of this paper and is proved in Section \ref{s3})
shows that the example presented above is essentially the only example of a stationary process
whose correlations over combinatorial lines deviate from what is expected.
\begin{thm} \label{thm:dich_lines}
Let $k\meg 2$ be an integer, and let $\ee, \sigma, \eta>0$ be such that
\begin{equation} \label{e1.7}
\ee \mik 1 - \frac{1}{2k}, \ \ \sigma \mik \frac{\ee^{k - 1}}{2k} \ \text{ and } \  \,
\eta \mik \frac{\sigma}{4^{k-1}}.
\end{equation}
Also let $A$ be a set with $|A| = k$, let $n\meg k$ be an integer, and let
$\langle D_t: t \in A^n\rangle$ be an $\eta$-stationary process in a probability space
$(\Omega,\mathcal{F}, \mathbb{P})$ such that\, $|\mathbb{P}(D_t)-\ee|\mik\eta$ for every $t\in A^n$. Then, either
\begin{enumerate}
\item [(i)] for every combinatorial line $L$ of $A^n$ and every nonempty $G\subseteq L$ we have
\begin{equation} \label{e1.8}
\Big|\mathbb{P}\Big( \bigcap_{t \in G} D_t \Big) - \ee^{|G|}\Big|\mik \sigma,
\end{equation}
\item [(ii)] or $\langle D_t:t\in A^n\rangle$ correlates with a ``structured" stochastic process; precisely,
there exist a nonempty subset $\Gamma$ of $A$, $\beta\in A \setminus \Gamma$ and a stochastic process
$\langle S_t: t \in A^n\rangle$ in $(\Omega,\mathcal{F},\mathbb{P})$ such that the following are satisfied.
\begin{enumerate}
\item [(a)] For every $t\in A^n$ we have $S_t=\bigcap_{\alpha\in \Gamma}  E^\alpha_t$ where for every $\alpha\in \Gamma$
the process $\langle E_t^\alpha: t \in A^n\rangle$ is $(\alpha,\beta)$-insensitive.
\item [(b)] For every $t\in A^n$ which contains $\beta$ we have
\begin{equation} \label{e1.9}
\mathbb{P}(S_t) \meg \frac{\ee^{k-1}}{4k}\ \text{ and } \ \, \mathbb{P}(D_t\, |\, S_t) \meg \ee + \frac{\sigma}{4^{k-1}}.
\end{equation}
\end{enumerate}
\end{enumerate}
\end{thm}
Theorem \ref{thm:dich_lines} is a new result whose most surprising feature is perhaps the fact that the conditional probability
$\mathbb{P}(D_t\, |\, S_t)$ depends \textit{linearly} on the parameter $\sigma$. As it is expected by Theorem \ref{PHJ1.1},
this information can in turn be used to prove the density Hales--Jewett theorem. We present this proof and we discuss in
detail its quantitative aspects in Section~\ref{s4}. At this point we simply mention that it is a step towards obtaining
primitive recursive bounds for the density Hales--Jewett numbers.

We also note that an infinitary structural result for stationary processes has been obtained by Austin
in \cite[Theorem 6.2]{Au}; it is somewhat different than Theorem \ref{thm:dich_lines}, but it provides
additional information for a certain subclass of stationary processes\footnote{This is related to the
notion of \emph{satedness}---see \cite{Au} for details.}. (See also \cite{Tao} for a finitary and
effective version of Austin's approach.)
\subsection{Correlations over arbitrary sets} \label{ss1.3}

Beyond its combinatorial consequences, Theorem \ref{thm:dich_lines} is also the starting point of the analysis
of correlations of stochastic processes over arbitrary nonempty subsets of discrete hypercubes.
This analysis leads to an answer to Problem \ref{prob1}, and it is presented in the second part of this
paper\footnote{The two parts are largely independent of each other and can be read separately.}
consisting of Sections \ref{s5}--\ref{s8}. It can be seen as a natural---though not quite straightforward---generalization
of the study of correlations over combinatorial lines. Specifically, there are two notable differences.

Firstly, the argument relies on the notion of the \textit{type}, a Ramsey-theoretic invariant which was introduced
in \cite{DKT2} and encodes the ``geometry" of a nonempty subset of a discrete hypercube. The definition of this
invariant is recalled in Section~\ref{s5}, and it is crucially used in order to extend the notion of stationarity in this
more general context (Definition \ref{stationarity6} in the main text).

Secondly, the ``structured" process which appears in part (ii.a) of Theorem \ref{thm:dich_lines}
depends upon the type of the set $G$ one is looking at in part (i). This dependence is controlled
by another invariant---the \textit{separation index}---which is introduced in Section \ref{s6}.
In particular, for correlations over sets which have the smallest possible separation index we have the
exact analogue of Theorem \ref{thm:dich_lines} (Theorem \ref{t7.2} in the main text); however, the analogy
breaks down at this point and the ``structured" process which appears in part (ii.a) becomes more involved
as the separation index increases (see Theorem \ref{t8.5}).

\subsection{Outline of the argument} \label{ss1.4}

The proof of Theorem \ref{thm:dich_lines} proceeds in two steps. In the first step and assuming that
part (i) does not hold true, we select a subset $B$ of $A$ such that for every variable word $v$
over $A$ of length $n$ and every nonempty proper subset $\Sigma$ of~$B$ the events
$\langle D_{v(\alpha)}:\alpha\in \Sigma\rangle$ are essentially independent, yet the joint probability
of $\langle D_{v(\alpha)}:\alpha\in B\rangle$ deviates from the expected value.
We emphasize that this selection is possible because the process $\langle D_t:t\in A^n\rangle$ is stationary.
The second step, which is the combinatorial heart of the matter, is to convert the irregularity of the
correlations of $\langle D_t:t\in A^n\rangle$ into correlation with a single structured process.
This is achieved by taking advantage of the uniform behavior of $\langle D_{v(\alpha)}:\alpha\in B\rangle$
as $v$ varies over all variable words over $A$ of length $n$, and by carefully using the ``projections"
$t^{3\to 1}$ and $t^{3\to 2}$ described in Example \ref{examp-intro} as well as their natural generalizations.

The argument for the case of correlations over arbitrary sets follows the same outline,
though the details are---as expected---more complicated. We comment on the differences of
the proof of the general case in Sections \ref{s7} and \ref{s8}.

\subsection{Acknowledgments} \label{ss1.5} We would like to thank the anonymous reviewers for carefully reading
the paper and for several helpful suggestions.

%----------------Combinatorial background---------------------%

\section{Combinatorial background}

\numberwithin{equation}{section} \label{s2}

\subsection{ \ } \label{ss2.1}

By $\nn=\{0,1,2,\dots\}$ we denote the set of all natural numbers, and for every positive integer $n$
we set $[n]\coloneqq \{1,\dots,n\}$. For every set $X$ we denote by $|X|$ its cardinality; moreover, for every
subset $A$ of $X$ by $A^{\complement}$ we shall denote the complement of~$A$, that is,
$A^{\complement}\coloneqq \{x\in X: x\notin A\}$.

\subsection{Definitions} \label{ss2.2}
Let $A$ denote a finite set with $|A|\meg 2$.

\subsubsection{ \ } \label{sss2.2.1}

As in \eqref{e1.1}, for every positive integer $n$ by $A^n$ we denote the Cartesian product of
$n$ many copies of $A$; we view $A^n$ as the set of all sequences of length~$n$ having values in $A$. Also let
$\emptyset$ denote the empty sequence, set $A^0\coloneqq\{\emptyset\}$, and let
\begin{equation} \label{e2.1}
A^{<\nn}\coloneqq \bigcup_{n\in \nn} A^n
\end{equation}
denote the set of all finite (possibly empty) sequences in $A$. For every $t,s\in A^{<\nn}$ by $t^{\con}s$ we denote
the \emph{concatenation} of $t$ and $s$; notice, in particular, that if $t\in A^n$ and $s\in A^m$ for some $n,m\in\nn$,
then $t^{\con}s\in A^{n+m}$.

\subsubsection{Variable words} \label{sss2.2.2}

Let $n,m$ be positive integers, and fix a set $\{x_1,\dots,x_m\}$ which is disjoint from $A$;
we view $\{x_1,\dots,x_m\}$ as a set of variables. An \emph{$m$-variable word over $A$ of\, length $n$} is a
finite sequence $v$ of length $n$ having values in $A\cup\{x_1,\dots,x_m\}$ such that: (1) for every $i\in[m]$ the letter
$x_i$ appears in $v$ at least once, and (2) if $m\meg 2$, then for every $i,j\in [m]$ with $i<j$ all appearances of $x_i$
in $v$ precede all appearances of $x_j$. If $v$ is an $m$-variable word over $A$ of length $n$ and $\alpha_1,\dots,\alpha_m\in A$,
then by $v(\alpha_1,\dots,\alpha_m)$ we denote the unique element of $A^n$ which is obtained by replacing every appearance
of $x_i$ in $v$ with $\alpha_i$ for every $i\in [m]$. (For example, if $A=\{\alpha,\beta,\gamma\}$ and
$v=(x_1,\gamma,x_2,x_2,\beta,x_3)$, then $v(\beta,\alpha,\gamma)=(\beta,\gamma,\alpha,\alpha,\beta,\gamma)$.)

\subsubsection{Combinatorial spaces and canonical isomorphisms} \label{sss2.2.3}

A \emph{combinatorial space} of~$A^{<\nn}$ is a subset $V$ of $A^{<\nn}$ of the form
\begin{equation} \label{e2.2}
V = \{ v(\alpha_1,\dots,\alpha_m) : \alpha_1,\dots, \alpha_m\in A\}
\end{equation}
where $m$ is a positive integer and $v$ is an $m$-variable word over $A$ of length $n$ for some positive integer $n$
(in particular, we have $V\subseteq A^n$.) Notice that $m,v$~and~$n$ are unique since $|A|\meg 2$; the (unique) positive
integer $m$ is called the \emph{dimension} of $V$ and is denoted by $\dim(V)$. Also observe that the $1$-dimensional
combinatorial spaces are precisely the combinatorial lines already mentioned in the introduction. Finally, if $V_1$
and $V_2$ are two combinatorial spaces of $A^{<\nn}$, then we say that $V_1$ is a \emph{combinatorial subspace} of
$V_2$ provided that $V_1\subseteq V_2$.

We view an $m$-dimensional combinatorial space $V$ as a ``copy" of $A^m$ inside~$A^{<\nn}$, and we will identify $V$
with $A^m$ for most practical purposes. To this end, we introduce the following definition.
\begin{defn} \label{d2.1}
Let $A$ be a finite set with $|A|\meg2$, and let $V$ be a combinatorial space of $A^{<\nn}$.
Set $m\coloneqq \mathrm{dim}(V)$ and let $v$ be the unique $m$-variable word over~$A$ which generates $V$
via formula \eqref{e2.2}. The \emph{canonical isomorphism associated with~$V$} is the bijection
$\mathrm{I}_V\colon A^m\to V$ defined by the rule
\begin{equation} \label{e2.3}
\mathrm{I}_V\big((\alpha_1,\dots,\alpha_m)\big)=v(\alpha_1,\dots,\alpha_m).
\end{equation}
for every $(\alpha_1 ,\dots,\alpha_m)\in A^m$.
\end{defn}
Note that canonical isomorphisms preserve combinatorial subspaces and their dimension; precisely, if $V$ is an $m$-dimensional
combinatorial space of $A^{<\nn}$ and $W\subseteq A^m$, then $W$ is a combinatorial subspace of $A^m$ with $\dim(W)=\ell$
if and only if $\mathrm{I}_V(W)$ is a combinatorial subspace of $V$ with $\dim\big(\mathrm{I}_V(W)\big)=\ell$.
For an exposition of the properties of canonical isomorphisms we refer to \cite[Section 1.3]{DK}.

\subsection{Colorings of combinatorial lines} \label{ss2.3}

We will need the following special case\footnote{Actually, the Graham--Rothschild theorem refers
to \emph{parameter words}, a concept which is slightly different from the notion of a variable word.
However, for colorings of combinatorial lines the difference between the two frameworks is minor.}
of the Graham--Rothschild theorem \cite{GR}. The corresponding primitive recursive bounds are taken from \cite{Ty}.
\begin{prop} \label{p2.2}
For every triple $k,m,r$ of positive integers with $k\meg 2$ there exists a positive integer $N$ with the following property.
For every set $A$ with $|A|=k$, every combinatorial space $V$ of $A^{<\nn}$ with $\mathrm{dim}(V)\meg n$ and every
$r$-coloring of the set of all combinatorial lines of\, $V$ there exists an $m$-dimensional combinatorial subspace $W$ of\, $V$
such that the set of all combinatorial lines of\, $W$ is monochromatic. The least positive integer $N$ with this property
is denoted by $\mathrm{GRL}(k,m,r)$.

Moreover, the numbers $\mathrm{GRL}(k,m,r)$ are upper bounded by a primitive recursive function belonging to the class
$\mathcal{E}^5$ of Grzegorczyk's hierarchy.
\end{prop}
For a discussion of Grzegorczyk's hierarchy of primitive recursive functions and its role in analyzing the bounds associated
with various results in Ramsey theory we refer to \cite[Appendix A]{DK}.

%-----------Correlations over combinatorial lines-------------%

\section{Correlations over combinatorial lines} \label{s3}

\numberwithin{equation}{section}

In this section we give the proof of Theorem \ref{thm:dich_lines}. As we have noted in the introduction,
the argument (which also pertains to the proofs of Theorems \ref{t7.2} and \ref{t8.5}) can be roughly summarized
by saying that higher order correlations of a process can be converted into correlation with a single structured process.
Perhaps the most transparent instance of this fact is the proof of Proposition \ref{thm:step_cor_lines_simple} below.

We begin with some preliminary steps, including a discussion of some basic properties of stationary processes.

\subsection{Stationarity} \label{ss3.1}

We have already noted that the Graham--Rothschild theorem (more precisely, Proposition \ref{p2.2})
implies that stationary processes are the building blocks of arbitrary processes. In particular, we have the following fact.
The proof is straightforward.
\begin{fact} \label{f3.1}
Let $k\meg 2$ be an integer, and let $A$ be a set with $|A|=k$. Also let $0<\eta\mik 1$, and let $n,m$ be positive integers such that
\begin{equation} \label{e3.1}
n\meg \mathrm{GRL}\big( k,m,\left\lceil 1/\eta\right\rceil^{2^k-1}\big).
\end{equation}
Then for every stochastic process $\langle D_t:t\in A^n\rangle$ in a probability space $(\Omega,\mathcal{F},\mathbb{P})$
there exists an $m$-dimensional combinatorial subspace $V$ of $A^n$ such that the process $\langle D_{\mathrm{I}_V(s)}:s\in A^m\rangle$
$($namely, the restriction of $\langle D_t:t\in A^n\rangle$ to $V$$)$ is $\eta$-stationary.
\end{fact}
The following lemma shows that one can upgrade the estimate in \eqref{e1.5} and stabilize the joint distribution of certain
boolean combinations of the events of a stationary processes. (Here, and in the rest of this paper, we follow that convention that the
intersection of an empty family of events of a probability space $(\Omega,\mathcal{F},\mathbb{P})$ is equal to the sample space $\Omega$.)
\begin{lem} \label{lem:stasionarity_lines}
Let $A$ be a finite set with $|A|\meg2$, let $n$ be a positive integer, let $\eta>0$, and let $\langle D_t:t \in A^n\rangle$
be an $\eta$-stationary stochastic process in a probability space $(\Omega, \mathcal{F}, \mathbb{P})$. Then for every pair
$\Gamma_1,\Gamma_2\subseteq A$ with $\Gamma_1\cap\Gamma_2=\emptyset$ and every pair $v_1,v_2$ of variable words
over $A$ of length $n$ we have
\begin{equation} \label{e3.2}
\Big|\mathbb{P}\Big( \bigcap_{\alpha \in \Gamma_1} D_{v_1(\alpha)} \cap \bigcap_{\alpha\in\Gamma_2} D_{v_1(\alpha)}^{\complement}\Big) -
\mathbb{P}\Big( \bigcap_{\alpha\in\Gamma_1} D_{v_2(\alpha)} \cap \bigcap_{\alpha\in \Gamma_2}
 D_{v_2(\alpha)}^{\complement}\Big) \Big| \mik 2^{|\Gamma_2|} \eta.
\end{equation}
\end{lem}
\begin{proof}
Let $\Gamma_1,\Gamma_2, v_1,v_2$ be as in the statement of the lemma. Then, using the inclusion--exclusion formula, we have
\[\begin{split}
& \Big|\mathbb{P}\Big( \bigcap_{\alpha\in \Gamma_1} D_{v_1(\alpha)} \cap \bigcap_{\alpha\in\Gamma_2} D_{v_1(\alpha)}^{\complement}\Big) -
  \mathbb{P}\Big(\bigcap_{\alpha\in \Gamma_1} D_{v_2(\alpha)} \cap \bigcap_{\alpha\in \Gamma_2} D_{v_2(\alpha)}^{\complement}\Big) \Big| \\
& = \Big|\mathbb{P}\Big( \bigcap_{\alpha\in \Gamma_1} D_{v_1(\alpha)} \cap \big(\bigcup_{\alpha\in\Gamma_2} D_{v_1(\alpha)}\big)^{\complement}\Big) -
  \mathbb{P}\Big( \bigcap_{\alpha\in \Gamma_1} D_{v_2(\alpha)} \cap \big(\bigcup_{\alpha\in \Gamma_2} D_{v_2(\alpha)} \big)^{\complement}\Big) \Big| \\
& \mik \sum_{\Gamma\subseteq \Gamma_2} \Big|\mathbb{P}\Big( \bigcap_{\alpha\in \Gamma_1} D_{v_1(\alpha)}  \cap
  \bigcap_{\alpha\in \Gamma} D_{v_1(\alpha)}\Big) - \mathbb{P}\Big( \bigcap_{\alpha\in\Gamma_1} D_{v_2(\alpha)} \cap
	\bigcap_{\alpha\in \Gamma} D_{v_2(\alpha)} \Big) \Big| \stackrel{\eqref{e1.5}}{\mik} 2^{|\Gamma_2|} \eta
\end{split}\]
and the proof is completed.
\end{proof}
\begin{rem} \label{rem3.3}
We notice that the assumption in Theorem \ref{thm:dich_lines} that $|\mathbb{P}(D_t)-\ee|\mik \eta$ for every $t\in A^n$
follows from $\eta$-stationarity provided that the dimension $n$ is sufficiently large. Indeed, let $A, n$ and
$\langle D_t:t \in A^n\rangle$ be as in Theorem \ref{thm:dich_lines}; clearly, we have~$n\meg |A|$.
We select $t_0\in A^n$ such that every $\alpha\in A$ appears in $t_0$ at least once (this selection is possible
since~$n\meg |A|$), and we set $\ee\coloneqq \max\{\mathbb{P}(D_{t_0}),\eta\}>0$. Note that for every $t\in A^n$
there exist two variable words $v_1,v_2$ over $A$ of length $n$ and $\alpha\in A$ such that $t=v_1(\alpha)$ and
$t_0=v_2(\alpha)$. Invoking \eqref{e1.5}, we conclude that $|\mathbb{P}(D_t)-\ee|\mik \eta$.
\end{rem}

\subsection{Insensitivity} \label{ss3.2}

Let $A$ be a finite set with $|A|\meg 2$, let $n$ be a positive integer,
and let $\alpha,\beta\in A$ with $\alpha\neq \beta$. As in Example \ref{examp-intro}, for every $t\in A^n$
let $t^{\beta \to \alpha}$ denote the unique element of $(A\setminus\{\beta\})^n$ which is obtained by replacing
every appearance of $\beta$ in $t$ with $\alpha$. We will use this operation in order to produce insensitive processes.
To this end, we will need the following elementary (though crucial) fact. Its proof is straightforward.
\begin{fact} \label{f3.4}
Let $A$ be a finite set with $|A|\meg 2$, let $n$ be a positive integer, and let $\alpha,\beta\in A$ with $\alpha\neq \beta$.
Then the map $A^n\ni t\mapsto t^{\beta \to \alpha}\in (A\setminus\{\beta\})^n$ is a projection; that is, for every $t\in A^n$
which does not contain $\beta$ we have that $t^{\beta \to \alpha}=t$. Moreover, if\, $t,s\in A^n$ are $(\alpha,\beta)$-equivalent,
then $t^{\beta \to \alpha}=s^{\beta \to \alpha}$.
\end{fact}

\subsection{Pseudorandomness, supercorrelation, subcorrelation} \label{ss3.3}

Let $E_1,\dots, E_\ell$ be measurable events in a probability space with equal probability $\ee>0$,
and observe that the joint probability of $E_1,\dots, E_\ell$ can be naturally categorized according to
whether it is greater than, less than, or almost equal to the expected value $\ee^\ell$.
As expected, our analysis depends on this trichotomy, and as such, it is convenient to introduce the following definition.
\begin{defn} \label{defn:correlation_lines}
Let $A$ be a finite set with $|A|\meg2$, let $n \meg |A|$ be an integer, and let $0<\eta,\ee\mik 1$.
Also let $\langle D_t: t \in A^n \rangle $ be an $\eta$-stationary process in a probability space
$(\Omega, \mathcal{F}, \mathbb{P})$ such that\, $|\mathbb{P} ( D_t ) - \ee |\mik \eta $ for every $t\in A^n$.
Finally, let\, $\Gamma\subseteq A$ be nonempty, and let\, $\theta\meg 0$.
\begin{enumerate}
\item[(1)] \emph{(Pseudorandomness)} We say that\, $ \langle D_t: t \in A^n \rangle $ is
\emph{$(\Gamma, \theta)$-pseudorandom} provided that
$\big| \mathbb{P}\big(\bigcap_{ \alpha \in \Gamma } D_{v(\alpha)}\big) - \ee^{|\Gamma|}\big| \mik  \theta$
for every variable word $v$ over $A$ of length $n$.
\item[(2)] $($\emph{Supercorrelation}$)$ We say that\, $ \langle D_t: t \in A^n \rangle $ is
\emph{$(\Gamma, \theta)$-supercorrelated} provided that
$\mathbb{P}\big( \bigcap_{ \alpha \in \Gamma } D_{v( \alpha ) }\big)\meg \ee^{|\Gamma|} + \theta $
for every variable word $v$ over $A$ of length $n$.
\item[(3)] \emph{(Subcorrelation)} We say that\, $ \langle D_t: t \in A^n \rangle $ is
\emph{$(\Gamma, \theta)$-subcorrelated} provided that
$\mathbb{P}\big(\bigcap_{ \alpha \in \Gamma } D_{v( \alpha ) }\big)\mik \ee^{|\Gamma|} - \theta $
for every variable word $v$ over $A$ of length $n$.
\end{enumerate}
\end{defn}
We have the following fact.
\begin{fact} \label{pr:correlations}
Let $A, n, \eta,\ee$ and $\langle D_t: t \in A^n \rangle $ be as in Definition \emph{\ref{defn:correlation_lines}}.
Also let\, $\Gamma\subseteq A$ be nonempty, and let $\theta\meg \eta$. Then one of the following holds true.
\begin{enumerate}
\item[(i)] The process $\langle D_t: t \in A^n \rangle $ is $( \Gamma, \theta)$-pseudorandom.
\item[(ii)] The process $\langle D_t: t \in A^n \rangle $ is $( \Gamma, \theta - \eta)$-supercorrelated.
\item[(iii)] The process $\langle D_t: t \in A^n \rangle $ is $( \Gamma, \theta -\eta)$-subcorrelated.
\end{enumerate}
\end{fact}
\begin{proof}
Assume that (i) does not hold true, that is, there is a variable word $v$ over $A$ of length $n$ such that
either $\mathbb{P} \big( \bigcap_{ \alpha \in \Gamma } D_{ v( \alpha ) } \big) \meg \ee^{|\Gamma|} + \theta $,
or $\mathbb{P} \big( \bigcap_{ \alpha \in \Gamma } D_{ v( \alpha ) } \big) \mik \ee^{|\Gamma|} - \theta $.
Invoking the $\eta$-stationarity of $ \langle D_t: t \in A^n \rangle $, we see that the first alternative yields part (ii),
while the second alternative yields part (iii).
\end{proof}
We are ready to state the main result in this subsection.
\begin{prop} \label{thm:step_cor_lines_simple}
Let $A$ be a finite set with $|A|\meg2$, and let $n\meg |A|$ be an integer.
Also let $0<\eta,\ee\mik 1$, and let $ \langle D_t: t \in A^n \rangle $ be an $\eta$-stationary
stochastic process in a probability space $ ( \Omega, \mathcal{F}, \mathbb{P} ) $ such that\,
$|\mathbb{P} ( D_t ) - \ee | \mik \eta $ for every $t\in A^n$.
Finally, let $\theta, \sigma\meg 0$, let\, $\Gamma\subseteq A$ be nonempty, and let
$\beta \in A \setminus \Gamma$. Assume that $ \langle D_t: t \in A^n \rangle $
is $ ( \Gamma, \theta ) $-pseudorandom, and set $p\coloneqq | \Gamma | $. Then there exists
a stochastic process $ \langle S_t: t \in A^n \rangle $ in $ ( \Omega, \mathcal{F}, \mathbb{P} ) $
with the following properties.
\begin{enumerate}
\item [(i)] For every $t\in A^n$ we have $S_t=\bigcap_{\alpha \in \Gamma }  E^\alpha_t$
where for every $\alpha \in \Gamma$ the process $\langle E_t^\alpha:t \in A^n\rangle$ is $(\alpha, \beta)$-insensitive.
\item[(ii)] For every $t\in A^n$ which does not contain $\beta$ and every $\alpha\in\Gamma$
we have $E_t^\alpha = D_t$. $($Thus, $S_t = D_t$ for every $t\in A^n$ which does not contain $\beta$.$)$
\item[(iii)] For every $t\in A^n$ which contains $\beta$ we have $|\mathbb{P} ( S_t ) - \ee^p | \mik  \theta $.
\item [(iv)] If\, $ \langle D_t: t \in A^n \rangle $  is $( \Gamma \cup \{\beta\}, \sigma )$-supercorrelated, then
for every $t\in A^n$ which contains $\beta$ we have
\begin{equation} \label{e3.3}
\mathbb{P}(D_t\, |\, S_t) \meg \ee \Big( 1 + \frac{ \sigma \ee ^ {-1} - \theta }{ \ee ^ p + \theta } \Big).
\end{equation}
\item [(v)] If\, $\langle D_t: t \in A^n \rangle $ is $( \Gamma \cup \{\beta\}, \sigma )$-subcorrelated, then
for every $t\in A^n$ which contains $\beta$ we have
\begin{equation} \label{e3.4}
\mathbb{P} ( D_t\, |\, S_t ) \mik \ee \Big( 1 - \frac{ \sigma \ee ^ {-1} - \theta }{ \ee ^ p - \theta } \Big).
\end{equation}
\end{enumerate}
\end{prop}
\begin{proof}
We first  observe that the conditions in parts (i) and (ii) completely determine the stochastic process
$\langle E_t^\alpha:t \in A^n \rangle$ for every $\alpha\in \Gamma$. However, it is possible to give an
alternative (and more intrinsic) definition of these processes which facilitates the proofs of parts (iii)--(v)
and it is easier to generalize when considering correlations over more complicated sets (see, in particular,
Sections \ref{s7} and \ref{s8}). More precisely, notice that, by Fact \ref{f3.4}, we have $E^\alpha_t = D_{t^{\beta \to \alpha}}$
for every $t\in A^n$ and every $\alpha\in \Gamma$. We will also need the following important property of this
construction. For every $t\in A^n$ which contains $\beta$ let $v_t$ denote the unique variable word over
$A\setminus\{\beta\}$ of length~$n$ which is obtained by replacing every appearance of $\beta$ in $t$
with the variable $x$, and note that $t=v_t(\beta)$ and $t^{\beta \to \alpha} = v_t(\alpha)$
for every $\alpha\in \Gamma$. Consequently, for every $t\in A^n$ which contains $\beta$ we have
\begin{equation} \label{e3.5}
S_t = \bigcap_{\alpha \in \Gamma} D_{v_t (\alpha)} \  \text{ and } \
D_t \cap S_t = \bigcap_{\alpha \in \Gamma \cup \{\beta\}} D_{v_t (\alpha)}.
\end{equation}

After this preliminary discussion, we are ready to proceed to the rest of the proof. Part (iii) follows immediately
by the first identity in \eqref{e3.5} and our assumption that the process $\langle D_t: t \in A^n \rangle $
is $(\Gamma, \theta )$-pseudorandom.

For part (iv), assume that $ \langle D_t: t \in A^n \rangle $  is $( \Gamma \cup \{\beta\}, \sigma )$-supercorrelated.
Fix $t\in A^n$ which contains $\beta$. By the second identity in \eqref{e3.5} and the supercorrelation assumption,
we see that $\mathbb{P}(D_t \cap S_t) \meg \ee^{p+1} + \sigma$;
on the other hand, by part (iii), we have $\mathbb{P}(S_t) \mik \ee^p + \theta$. Therefore,
\[ \mathbb{P}( D_t\, |\, S_t) = \frac{\mathbb{P}(D_t \cap S_t)}{\mathbb{P}(S_t)} \meg \frac{\ee^{p+1} + \sigma}{\ee^p + \theta} =
\ee \Big( 1 + \frac{ \sigma \ee ^ {-1} - \theta }{ \ee ^ p + \theta } \Big)\]
as desired.

Finally, assume that $ \langle D_t: t \in A^n \rangle $ is  $( \Gamma \cup \{\beta\}, \sigma )$-subcorrelated,
and fix $t\in A^n$ which contains $\beta$. As above, using the second identity in \eqref{e3.5}
and the subcorrelation assumption, we obtain that  $\mathbb{P}(D_t \cap S_t) \mik \ee^{p+1} - \sigma$.
By part (iii), we have $\mathbb{P}(S_t) \meg \ee^p - \theta$, and so,
\[\mathbb{P}( D_t\, |\, S_t) = \frac{\mathbb{P}(D_t \cap S_t)}{\mathbb{P}(S_t)}
\mik \frac{\ee^{p+1} - \sigma}{\ee^p - \theta} =  \ee \Big( 1 - \frac{ \sigma \ee ^ {-1} - \theta }{ \ee ^ p - \theta } \Big). \]
The proof is completed.
\end{proof}
\begin{rem} \label{rem3.8}
Observe that the variable word $v_t$ defined in the proof of Proposition~\ref{thm:step_cor_lines_simple} is not typical
since it does not contain $\beta$. Nevertheless, because stationarity is a \textit{global} property, it is possible
to have information for the correlation of $\langle D_{v_t(\alpha)}:\alpha\in \Gamma\rangle$. This fact (namely,
the necessity to understand the correlations of $\langle D_t:t\in A^n\rangle$ over sparse sets of combinatorial lines)
is rather subtle and appears to be a genuine obstacle for extending Theorem \ref{thm:dich_lines} to not necessarily
stationary processes.
\end{rem}
\begin{rem}[Extreme cases] \label{rem-extr}
The extreme cases in Proposition \ref{thm:step_cor_lines_simple} are: (a)~``$\theta=0$" and ``$\sigma=\ee^p-\ee^{p+1}$"
if the stochastic process $\langle D_t: t \in A^n \rangle $ is supercorrelated, and (b)~``$\theta=0$" and ``$\sigma=\ee^{p+1}$"
if $\langle D_t: t \in A^n \rangle $ is subcorrelated. In the first case we have that $\mathbb{P} ( D_t\, |\, S_t ) = 1$
for every $t\in A^n$ containing $\beta$, which is clearly equivalent to saying that $S_t\subseteq D_t$.
Examples of stochastic processes of this form can be obtained by modifying (in a straightforward way) Example \ref{examp-intro}.
At~the other extreme, we see that $\mathbb{P} ( D_t\, |\, S_t ) = 0$ for every $t\in A^n$ which contains $\beta$.
In contrast to the previous case, this phenomenon cannot occur if the dimension $n$ is sufficiently large;
this is a consequence of Theorem \ref{PHJ1.1}.
\end{rem}

\subsection{Proof of Theorem \ref{thm:dich_lines}} \label{ss3.4}

We begin by introducing a finite sequence $(\theta_p)_{p=0}^k$ of positive reals defined by the rule
\begin{equation} \label{e3.6}
\left\{ \begin{array} {l}
\theta_0=0, \ \theta_1=\eta, \\
\theta_p = 4^{p-k} \sigma \text{ if } p\in\{2,\dots, k\}.
\end{array}  \right.
\end{equation}
(Note that, by \eqref{e1.7}, the sequence $(\theta_p)_{p=0}^k$ is increasing.) Next observe that if for every nonempty
$\Gamma\subseteq A$ the process $\langle D_t: t \in A^n\rangle$ is $(\Gamma,\theta_{|\Gamma|} )$-pseudorandom,
then part (i) of the theorem holds true. Therefore, we may assume that there exists nonempty $\Delta\subseteq A$
such that $\langle D_t: t \in A^n\rangle$ is \textit{not} $(\Delta,\theta_{|\Delta|} )$-pseudorandom. We fix a
nonempty subset $\Gamma_0$ of $A$ which satisfies this property and with minimal cardinality. (Notice, in particular,
that if $\Sigma$ is a nonempty proper subset of $\Gamma_0$, then $\langle D_t: t \in A^n\rangle$ is
$(\Sigma,\theta_{|\Sigma|} )$-pseudorandom.) Since the process $\langle D_t: t \in A^n\rangle$ is $\eta$-stationary
and $\theta_1=\eta$, we see that $|\Gamma_0| \meg 2$. We select $\beta\in \Gamma_0$, and we set
$\Gamma\coloneqq \Gamma_0 \setminus \{\beta\}$ and $p\coloneqq|\Gamma|$; observe that $1 \mik p \mik k-1$.
Set $\theta\coloneqq\theta_p$ and $\Theta\coloneqq\theta_{p+1}$.

By Fact \ref{pr:correlations} and our assumption that the stochastic process $\langle D_t: t \in A^n\rangle$ is not
$(\Gamma \cup \{\beta\}, \Theta )$-pseudorandom,  we see that either
\begin{enumerate}
\item[(A1)] $\langle D_t: t \in A^n\rangle$ is $(\Gamma \cup \{\beta\}, \Theta - \eta )$-supercorrelated,
\item[(A2)] or $\langle D_t: t \in A^n\rangle$ is $(\Gamma \cup \{\beta\}, \Theta - \eta )$-subcorrelated.
\end{enumerate}
We will show that in both cases part (ii) of the theorem holds true.
\medskip

\noindent \textsc{Case 1:} \textit{$\langle D_t: t \in A^n\rangle$ is $(\Gamma \cup \{\beta\}, \Theta - \eta )$-supercorrelated.}
By Proposition \ref{thm:step_cor_lines_simple} applied for ``$\sigma = \Theta - \eta$", there exists a process
$\langle S_t: t \in A^n\rangle$ which satisfies part (ii.a) of the theorem such that for every $t\in A^n$ which
contains $\beta$ we have
\begin{enumerate}
\item[(a)] $|\mathbb{P}(S_t)-\ee^p | \mik  \theta $, and
\item[(b)] $\mathbb{P}(D_t\, |\, S_t) \meg \ee \big( 1 + \frac{ \Theta\ee ^ {-1} - \eta\ee ^ {-1}  - \theta }{ \ee ^ p + \theta } \big)$.
\end{enumerate}
Therefore, by (a) above and the fact that $\theta\mik \sigma/4$, for every $t\in A^n$ which contains~$\beta$ we have
\[ \mathbb{P}(S_t)\meg \ee^p - \theta \meg \ee^{k-1}-\frac{\sigma}{4} \stackrel{\eqref{e1.7}}{\meg}
\frac{\ee^{k-1}}{4k} \]
while, by (b) and the fact that $\eta,\theta\mik \Theta/4$,
\[ \mathbb{P}( D_t\, |\, S_t )\meg \ee + \frac{\Theta - \eta - \ee \theta}{\ee^p + \theta}
\meg \ee + \frac{1}{2}( \Theta - \eta - \theta) \meg
\ee + \frac{\Theta}{4} \meg \ee + \frac{\theta_2}{4} = \ee + \frac{\sigma}{4^{k-1}}. \]
Thus, part (ii.b) of the theorem is also satisfied, as desired.
\medskip

\noindent \textsc{Case 2:} \textit{$\langle D_t: t \in A^n\rangle$ is $(\Gamma\cup\{\beta\},\Theta - \eta )$-subcorrelated.}
For every $t\in A^n$ and every $\alpha\in\Gamma$ set $\mathcal{E}^\alpha_t\coloneqq D_{ t^{\beta \to \alpha}}$.
(Recall that $t^{\beta \to \alpha}$ denotes the unique element of $(A\setminus\{\beta\})^n$ which is obtained by replacing every
appearance of $\beta$ in $t$ with~$\alpha$.) We select $\gamma\in\Gamma$, we set $B\coloneqq \Gamma\setminus\{\gamma\}$,
and for every $t\in A^n$ we define
\begin{equation}\label{e3.7}
S_t\coloneqq \Big(\bigcap_{\alpha\in B}\mathcal{E}^\alpha_t\Big)\cap (\mathcal{E}^\gamma_t)^\complement.
\end{equation}
(Recall that, by convention, $\bigcap_{\alpha\in B}\mathcal{E}^\alpha_t=\Omega$ if $B=\emptyset$.)
Clearly, the stochastic process $\langle S_t:t\in A^n \rangle$ satisfies part (ii.a) of the theorem.
As in the proof of Proposition~\ref{thm:step_cor_lines_simple}, for every $t\in A^n$ which contains $\beta$
by $v_t$ we denote the unique variable word over $A\setminus\{\beta\}$ of length $n$ which is obtained by replacing every
appearance of $\beta$ in $t$ with the variable $x$; recall that $t=v_t(\beta)$ and $t^{\beta \to \alpha} = v_t(\alpha)$
for every $\alpha\in \Gamma$. Consequently, for every $t\in A^n$ which contains $\beta$ we have
\begin{equation} \label{e3.8}
S_t = \Big(\bigcap_{\alpha \in B} D_{v_t (\alpha)}\Big) \setminus\Big(\bigcap_{\alpha \in \Gamma} D_{v_t (\alpha)}\Big)
\end{equation}
and
\begin{equation} \label{e3.9}
D_t \cap S_t = \Big(\bigcap_{\alpha \in B \cup \{\beta\}} D_{v_t (\alpha)}\Big)
\setminus\Big(\bigcap_{\alpha \in \Gamma \cup \{\beta\}} D_{v_t (\alpha)}\Big).
\end{equation}
Since $\langle D_t: t \in A^n\rangle$ is
\begin{enumerate}
\item[$\bullet$] $(\Gamma\cup\{\beta\},\Theta-\eta )$-subcorrelated,
\item[$\bullet$] $(B,\theta_{p-1})$-pseudorandom if $p>1$ (if $p=1$, then this is superfluous), and
\item[$\bullet$] $(\Gamma,\theta)$-pseudorandom and $(B\cup\{\beta\},\theta)$-pseudorandom,
\end{enumerate}
for every $t\in A^n$ which contains $\beta$ we have
\[ \mathbb{P}(S_t)\mik(\ee^{p-1}+\theta_{p-1})-(\ee^p-\theta) \ \text{ and } \
\mathbb{P}(D_t\cap S_t)\meg(\ee^p-\theta)-(\ee^{p+1}-\Theta+\eta). \]
Moreover, by \eqref{e1.7} and \eqref{e3.6}, we have $\theta+\theta_{p-1}\mik\ee^p$, $\theta\mik\Theta/4$ and $\theta_{p-1}+\eta\mik\Theta/4$.
Therefore, for every $t\in A^n$ which contains $\beta$
\[ \mathbb{P}(D_t\, |\, S_t) \meg \frac{\ee^p-\ee^{p+1}+\Theta-\theta-\eta}{\ee^{p-1}-\ee^p+\theta+\theta_{p-1}}
\meg\ee+\frac{\Theta-2\theta-\theta_{p-1}-\eta}{\ee^{p-1}-\ee^p+\theta+\theta_{p-1}}
\meg  \ee+\frac{\Theta}{4}\meg\ee+\frac{\sigma}{4^{k-1}}. \]
Finally, by \eqref{e1.7}, \eqref{e3.8} and the fact that $\langle D_t: t \in A^n\rangle$
is $(B,\theta_{p-1})$-pseudorandom and $(\Gamma,\theta)$-pseudorandom, we conclude that
\[\mathbb{P}(S_t)\meg\ee^{p-1}-\ee^p-\theta-\theta_{p-1}\meg\ee^{k-1}(1-\ee)-2\theta
\meg\frac{\ee^{k-1}}{2k}-2\theta\meg\frac{\ee^{k-1}}{4k}\]
for every $t\in A^n$ which contains $\beta$. The proof is completed.

%---------Proof of the density Hales--Jewett theorem----------%

\section{Proof of the density Hales--Jewett theorem}

\numberwithin{equation}{section} \label{s4}

\subsection{ \ } \label{ss4.1}

In this section we give a proof of the density Hales--Jewett theorem which is based on Theorem \ref{thm:dich_lines}.
We begin by recalling the combinatorial version of the density Hales--Jewett theorem. (The reader is advised to compare this version
with Theorem \ref{PHJ1.1} stated in the introduction.)
\begin{thm} \label{DHJ4.1}
For every integer $k\meg 2$ and every $0<\delta\mik 1$ there exists a positive integer $\mathrm{DHJ}(k,\delta)$ with the following
property. Let $A$ be a set with $|A|=k$, and let $n\meg \mathrm{DHJ}(k,\delta)$ be an integer. Then every $D\subseteq A^n$
with $|D|\meg \delta |A^n|$ contains a combinatorial line of $A^n$.
\end{thm}
There are several effective proofs\footnote{Another ergodic-theoretic proof was given in \cite{Au}.}
of Theorem \ref{DHJ4.1}; see \cite{DKT1,P2,Tao}. Despite this progress, the understanding of the behavior
of the density Hales--Jewett numbers $\mathrm{DHJ}(k,\delta)$ is rather poor. Indeed, the best known upper bounds
are obtained in \cite{P2} and have an Ackermann-type dependence with respect to $k$.

The proof of Theorem \ref{DHJ4.1} given below is based on a density increment strategy (a method introduced
by Roth \cite{Ro}) and follows the general scheme developed in~\cite{P2}. Its most important feature is the
quantitative improvement of a crucial step which appears (in various forms) in all known combinatorial proofs
of the density Hales--Jewett theorem. (We discuss this particular feature in Remark \ref{rem-improvement} below.)
The driving force behind this improvement is Theorem~\ref{thm:dich_lines}.

\subsubsection{Step 1: from dense subsets of discrete hypercubes to stochastic processes} \label{sss4.1.1}

Strictly speaking, this step is not an internal part of the proof of Theorem \ref{DHJ4.1}. However, it is conceptually
significant since it enables us to pass from dense sets to stochastic processes. This is essentially the content of
the following simple lemma whose proof can be found in \cite[Lemma 4]{DKT1}.
\begin{lem} \label{lem:dens_sections}
Let $k,m$ be positive integers with $k\meg 2$, let $0<\eta\mik 1$, let $A$ be a set with $|A|=k$,
and let $n$ be a positive integer such that
\begin{equation} \label{edhjnew1}
n\meg \frac{k^m\, m}{\eta}.
\end{equation}
Then for every $D\subseteq A^n$ there exist $\ell\in\{m,\dots,n-1\}$ and an $m$-dimensional combinatorial subspace
$V$ of $A^\ell$ such that for every $t\in V$ we have
\begin{equation} \label{edhjnew2}
\frac{|D_t|}{|A^{n-\ell}|}\meg \frac{|D|}{|A^n|}-\eta
\end{equation}
where $D_t = \{s\in A^{n-\ell}:t^\smallfrown s\in D\}$ denotes the section of $D$ at $t$.
\end{lem}
\begin{rem}
There is a more powerful probabilistic version of Lemma \ref{lem:dens_sections} which can be stated
as a concentration inequality and relies on properties of martingale difference sequences; see \cite[Theorem 1]{DKT3}.
See also \cite[Chapter 6]{DK} for a discussion on the role of this result in density Ramsey theory.
\end{rem}
\begin{rem}
Lemma \ref{lem:dens_sections} can be used to relate the numerical invariants $\mathrm{PHJ}(k,\ee)$ and $\mathrm{DHJ}(k,\delta)$
associated with the two versions of the density Hales--Jewett theorem. Indeed, notice that for every integer $k\meg 2$
and every $0<\theta<\ee\mik 1$ we~have
\begin{equation} \label{edhjnew3}
\mathrm{PHJ}(k,\ee) \mik \mathrm{DHJ}(k,\ee) \mik (\ee-\theta)^{-1}\cdot \mathrm{PHJ}(k,\theta)\cdot k^{\mathrm{PHJ}(k,\theta)}.
\end{equation}
\end{rem}

\subsubsection{Step 2: obtaining correlation with an insensitive set} \label{sss4.1.2}

We start by introducing the combinatorial analogue of the notion of an insensitive process. We note that this combinatorial
analogue in fact predates Definition \ref{insensitivity-intro}.
\begin{defn} \label{insensitivity-sets}
Let $A$ be a finite set with $|A|\meg2$, let $n$ be a positive integer, and let $\alpha, \beta\in A$ with $\alpha\neq \beta$.
\begin{enumerate}
\item[(1)] We say that a subset $\mathcal{E}$ of $A^n$ is \emph{$(\alpha, \beta)$-insensitive} if for
every $s,t\in A^n$ which are $(\alpha, \beta)$-equivalent we have that $t\in \mathcal{E}$ if and only if $s\in \mathcal{E}$.
\item[(2)] We say that a subset $\mathcal{E}$ of an $n$-dimensional combinatorial space $V$ of $A^{<\nn}$
is \emph{$(\alpha, \beta)$-insensitive in $V$} if\, $\mathrm{I}^{-1}_V(\mathcal{E})$ is $(\alpha, \beta)$-insensitive,
where $\mathrm{I}_V\colon A^n\to V$ denotes the canonical isomorphism associated with $V$.
\end{enumerate}
\end{defn}
The following lemma is the second step of the proof of Theorem~\ref{DHJ4.1}.
It is precisely in the proof of this step that Theorem \ref{thm:dich_lines} is applied.
\begin{lem} \label{lem:incre_DHJ}
Let $m\meg k\meg 2$ be integers, and let\, $0<\delta\mik 1$. Set
\begin{equation} \label{edhj1}
N=\mathrm{GRL}\big(k+1,m+1,\lceil 2(k+1)4^k \delta^{-(k+1)}\rceil^{(2^{k+1}-1)}\big)
\end{equation}
and let $n$ be a positive integer such that
\begin{equation} \label{edhj2}
n \meg  \frac{2(k+1)4^k}{\delta^{k+1}}\, (k+1)^{N}N .
\end{equation}
Let $A$ be a set with $|A|=k+1$, and let $D\subseteq A^n$ with $|D|\meg \delta |A^n|$. Then, either
\begin{enumerate}
\item [(i)] $D$ contains a combinatorial line of $A^n$, or
\item [(ii)] there exist $\beta\in A$, an $m$-dimensional combinatorial subspace $V$ of $A^n$
and a subset $\mathcal{S}$ of\, $V$ with the following properties.
\begin{enumerate}
\item [(a)] We have that $\mathcal{S} = \bigcap_{\alpha \in A\setminus\{\beta\}} \mathcal{E}^\alpha$
where for every $\alpha\in A\setminus\{\beta\}$ the set $\mathcal{E}^\alpha$ is $(\alpha,\beta)$-insensitive in $V$.
\item [(b)] We have
\begin{equation} \label{edhj3}
\frac{|\mathcal{S}|}{|V|} \meg \frac{\delta^{2k+1}}{(k+1)^2\,4^{k+2}} \ \text{ and } \
\frac{|D\cap \mathcal{S}|}{|\mathcal{S}|} \meg \delta + \frac{\delta^{k+1}}{(k+1)\,4^{k+1}}.
\end{equation}
\end{enumerate}
\end{enumerate}
\end{lem}
\begin{rem} \label{rem-improvement}
Lemma \ref{lem:incre_DHJ} improves upon two important quantitative aspects of what was known before.
Firstly, by Proposition \ref{p2.2}, the threshold on the dimension $n$ appearing in~\eqref{edhj2} is bounded
by a primitive recursive function which belongs to the class $\mathcal{E}^5$ of Grzegorczyk's hierarchy; in particular,
it is independent of the numbers $\mathrm{DHJ}(k,\delta)$.  Secondly, the increment of the density of the set $D$ obtained
in the second part of~\eqref{edhj3} depends polynomially on $\delta$; in order to appreciate this particular improvement
we recall that all previous proofs yield a density increment which has an Ackermann-type dependence with respect to $k$.
We also note that this quantity controls the number of iterations needed to be performed in order to prove
Theorem~\ref{DHJ4.1}, and as such it has significant impact on the behavior of the density Hales--Jewett numbers.
\end{rem}
\begin{proof}[Proof of Lemma \emph{\ref{lem:incre_DHJ}}]
We start by setting $\eta\coloneqq \frac{\delta^{k+1}}{2(k + 1)4^k}$.
By Lemma \ref{lem:dens_sections} and \eqref{edhj2}, there exist $\ell\in\{m,\dots,n-1\}$
and an $N$-dimensional combinatorial subspace $V_1$ of $A^\ell$ such that for every $t\in V_1$ we have
\begin{equation} \label{eq:020}
\frac{|D_t|}{|A^{n-\ell}|}\meg\frac{|D|}{|A^n|}-\eta.
\end{equation}
We view the set $A^{n-\ell}$ as a discrete probability measure equipped with the uniform probability measure which
we shall denote by $\mathbb{P}_1$. By Fact \ref{f3.1} and the choice of $N$ in \eqref{edhj1}, there exists an $(m+1)$-dimensional
combinatorial subspace $V_2$ of $V_1$ such that the process $\langle D_{\mathrm{I}_{V_2}(t)}:t\in A^{m+1}\rangle$
is $\eta\text{-stationary}$; consequently, by Remark \ref{rem3.3}, \eqref{eq:020} and the fact that $|D|\meg\delta|A^n|$,
there exists $\ee \meg \delta$ such that $|\mathbb{P}_1(D_t) - \ee | \mik \eta$ for every $t\in V_2$.

Now assume that part (i) does not hold true, that is, the set $D$ contains no combinatorial line of $A^n$.
This in turn implies that $\bigcap_{t \in L} D_t = \emptyset$ for every combinatorial line $L$ of $V_2$; in particular,
$\ee \mik 1 - \frac{1}{2(k + 1)}$. Next, set $\sigma\coloneqq \frac{\ee^{k+1}}{2(k+1)}$ and notice that
$\eta\mik \sigma/4^k$. Thus, by Theorem \ref{thm:dich_lines}, there exist a nonempty subset $\Gamma$ of $A$,
$\beta\in A \setminus \Gamma$ and a stochastic process  $\langle S_{\mathrm{I}_{V_2}(t)}: t \in A^{m+1}\rangle$
consisting of subsets of $A^{n-\ell}$ such that the following are satisfied.
\begin{enumerate}
\item[(a)] For every $t\in V_2$ we have $S_t = \bigcap _ { \alpha \in \Gamma}E^\alpha_t$ where for every $\alpha\in \Gamma$
the stochastic process $\langle E_{\mathrm{I}_{V_2}(t)}^\alpha:t \in A^{m+1}\rangle$ is $(\alpha, \beta)$-insensitive.
\item[(b)] For every $t\in V_2$ such that $\mathrm{I}_{V_2}^{-1}(t)$ contains $\beta$ we have
\begin{equation} \label{enew5.8}
\mathbb{P}_1(S_t) \meg \frac{\ee^k}{4(k+1)} \ \text{ and } \ \mathbb{P}_1(D_t\, |\, S_t) \meg \ee + \frac{\sigma}{4^k}.
\end{equation}
\end{enumerate}
By setting $E^\alpha_t=A^{n-\ell}$ for every $t\in V_2$ and every $\alpha\in A \setminus (\Gamma\cup \{\beta\})$,
we may assume that $\Gamma = A \setminus \{\beta\}$. Next, let $V_3$ denote the set of all $t\in V_2$ such that
$\mathrm{I}_{V_2}^{-1}(t)$ starts with $\beta$, and notice that $V_3$ is an $m$-dimensional combinatorial subspace of~$V_2$.
Also observe that property (a) above and \eqref{enew5.8} hold true for every $t\in V_3$.

With the process $\langle S_t:t\in V_3\rangle$ at our disposal the rest of the proof follows by a double counting argument
and an application of the first moment method. Indeed, let $\mathbb{P}_2$ and $\mathbb{P}_3$ denote the uniform probability
measures on $V_3$ and $V_3 \times A^{n-\ell}$ respectively.
Set $S\coloneqq \bigcup_{t\in V_3}\{t\} \times S_t\subseteq V_3 \times A^{n-\ell}$ and notice that, by \eqref{enew5.8},
\begin{equation} \label{eq:021}
\mathbb{P}_3(S)\meg \frac{\ee^k}{4(k+1)} \ \text{ and } \ \mathbb{P}_3( D\, | \, S)\meg \ee + \frac{\sigma}{4^k}.
\end{equation}
For every $s\in A^{n-\ell}$ let $S^s=\{t \in V_3 :t^\smallfrown s \in S \}$ and $D^s=\{t \in V_3 :t^\smallfrown s \in D \}$
denote the sections of $S$ and $D$ at $s$ respectively, and set
\[ B\coloneqq \Big\{s \in A^{n-\ell} : \mathbb{P}_2(S^s) \mik \frac{\ee^k\, \sigma}{2(k+1) 4^{k+1}}\Big\}
\ \text{ and } \ C\coloneqq \bigcup _ { s \in B } S^s\times\{s\}\subseteq S. \]
Noticing that $\mathbb{P}_3(C)\mik (\ee^k\, \sigma)/ (2(k+1) 4^{k+1})$, by \eqref{eq:021}, we obtain that
\begin{equation} \label{eq:022}
\mathbb{P}_3(C\, |\, S)\mik \frac{\sigma}{2 \cdot 4^k}.
\end{equation}
We thus have
\begin{eqnarray*}
\mathbb{P}_3(D\, |\, S \setminus C) & = & \frac{\mathbb{P}_3 (D \cap (S \setminus C))}{\mathbb{P}_3 (S \setminus C)}
   \meg \frac{\mathbb{P}_3 (D \cap (S \setminus C))}{\mathbb{P}_3 (S)} \\
& \meg & \mathbb{P}_3(D\, |\, S)-\mathbb{P}_3(C\, |\, S) \stackrel{\eqref{eq:021}, \eqref{eq:022}}{\meg} \ee + \frac{\sigma}{2\cdot 4^k}.
\end{eqnarray*}
Since
\[ \mathbb{P}_3(D\, |\, S \setminus C) = \sum_{s \in A^{n-\ell} \setminus B} \mathbb{P}_2(D^s\, |\, S^s) \cdot
\mathbb{P}_3 (S^s\times\{s\} \, |\, S\setminus C) \]
and $\sum_{s \in A^{n-\ell} \setminus B} \mathbb{P}_3 (S^s\times\{s\}\, |\, S \setminus C) = 1$, there exists
$s\in A^{n-\ell} \setminus B$ such that
\[ \mathbb{P}_2(D^s\, |\, S^s) \meg\ee + \sigma/ (2\cdot 4^k). \]
We set
\[ V\!\coloneqq V_3\times\{s\}, \ \, \mathcal{S}\!\coloneqq\! S \cap V \, \text{ and } \,
\mathcal{E}^\alpha\!\coloneqq\! \Big(\bigcup_{t\in V_3} \{t\} \times E^\alpha_t\Big) \cap V
\text{ for every } \alpha\in A\setminus \{\beta\}. \]
It is easy to see that with these choices the second part of the lemma is satisfied. The proof is completed.
\end{proof}

\subsubsection{Step 3: partitioning the insensitive set into combinatorial subspaces} \label{sss4.1.3}

The following lemma, which is proved in \cite[Lemma 8.2]{P2}, is the last step of the proof of Theorem \ref{DHJ4.1}.
\begin{lem} \label{tiling}
Let $k\meg 2$ be an integer, and assume that for every $0<\delta\mik 1$ the number $\mathrm{DHJ}(k,\delta)$ has been defined.

Then for every positive integer $m$ and every $0<\eta\mik 1$ there exists a positive integer
$\mathrm{Til}(k,m,\eta)$---which depends on the numbers $\mathrm{DHJ}(k,\delta)$---satisfying the following property.
Let $A$ be a set with $|A|=k+1$, let $n\meg \mathrm{Til}(k,m,\eta)$ be an integer, and let $\beta\in A$. Also let
$V$ be an $n$-dimensional combinatorial subspace of $A^{<\nn}$ and let $\mathcal{S}\subseteq V$ which is of the form
$\mathcal{S}=\bigcap_{\alpha\in A\setminus \{\beta\}}\mathcal{E}^\alpha$ where $\mathcal{E}^\alpha$ is
$(\alpha,\beta)$-insensitive in $V$ for every $\alpha\in A\setminus\{\beta\}$. Then there exists a $($possibly empty$)$
collection $\mathcal{W}$ of pairwise disjoint $m$-dimensional combinatorial subspaces of\, $V$ with
$\cup\mathcal{W}\subseteq \mathcal{S}$ and such that $|\mathcal{S}\setminus \cup\mathcal{W}|\mik \eta |V|$.
\end{lem}
Although the proof of Lemma \ref{tiling} given in \cite{P2} is quite natural, unfortunately it leads to a very bad
dependence of the numbers $\mathrm{Til}(k,m,\eta)$ on the numbers $\mathrm{DHJ}(k,\delta)$---see, \emph{e.g.}, \cite[Section 9]{P2}
for a discussion on this issue.

\subsubsection{Completion of the proof of Theorem \emph{\ref{DHJ4.1}}} \label{sss4.1.4}

Given Lemmas \ref{lem:incre_DHJ} and \ref{tiling}, the proof of Theorem \ref{DHJ4.1} follows easily by induction on $k$.
(The base case ``$k=2$" is a consequence of the classical Sperner theorem \cite{Sp}.) See, \emph{e.g.}, \cite[Chapter 8]{DK}
or \cite{P2} for detailed expositions.

\subsection{Comments} \label{ss4.2}

As alluded to earlier, Lemma \ref{lem:incre_DHJ} is a step towards obtaining primitive recursive bounds for the numbers
$\mathrm{DHJ}(k,\delta)$. It is clear that what is missing at this point is a quantitatively not wasteful proof
of Lemma \ref{tiling} (or a related variant). Although this will certainly require new ideas, it is likely that
this program will eventually lead to primitive recursive bounds for the numbers $\mathrm{DHJ}(k,\delta)$ belonging
to the class $\mathcal{E}^7$ of Grzegorczyk's hierarchy or slightly higher.

A disadvantage of this approach is that it relies on an analysis which is ``local" in nature because we assume stationarity.
It would be much more desirable if we had a ``global" structure theorem. Formulating and proving a ``global" theorem with
quantitative aspects comparable to that of Theorem \ref{thm:dich_lines} might lead to upper bounds for the numbers
$\mathrm{DHJ}(k,\delta)$ which are of tower-type; note that this would also improve the longstanding upper
bounds for the coloring version of the Hales--Jewett theorem obtained by Shelah \cite{Sh}.

However, even  tower-type upper bounds are rather unlikely to be anywhere close to optimal. Indeed, the best known lower bounds
for the numbers $\mathrm{DHJ}(k,\delta)$ are merely quasi-polynomial with respect to $\delta^{-1}$ (see \cite[Theorem 1.3]{P1}).

%---------The type of a subset of a discrete hypercube--------%

\section{The type of a subset of a discrete hypercube} \label{s5}

\numberwithin{equation}{section}

This is the first section of the second part of this paper which is devoted to the study of correlations
of stochastic processes over arbitrary nonempty subsets of discrete hypercubes. As we have pointed out in the introduction,
the analysis of these correlations relies, in a essentially way, on the notion of the \textit{type} of a nonempty
subset of $A^n$. This Ramsey-theoretic invariant was introduced in~\cite{DKT2}, though it can be traced\footnote{More
precisely, the results in \cite{FK1} concern colorings of variable words---this is a similar, but not identical, setting.}
to \cite{FK1}. We point out that for technical reasons (that will become transparent in Sections \ref{s6}, \ref{s7} and \ref{s8}),
we will work with nonempty tuples of distinct elements of hypercubes instead of nonempty finite sets. This is an
equivalent framework, but it does have some impact on our exposition when compared with that of~\cite{DKT2}.
With this machinery at our disposal, it is straightforward to extend the notions of stationarity, pseudorandomness,
supercorrelation and subcorrelation introduced in Definitions \ref{defn:stasionarity_lines} and
\ref{defn:correlation_lines} respectively; these extensions are presented in Subsection \ref{ss5.4}.

\subsection{The type of a nonempty tuple} \label{ss5.1}

Let $A$ be a finite set with $|A|\meg 2$, and let $n,p$ be positive integers with $p\mik |A|^n$.
Let $\mathbf{t}=(t_1,\dots, t_p)$ be a nonempty tuple (a~nonempty finite sequence) of distinct elements of $A^n$.

\subsubsection{ \ } \label{sss5.1.1}

If $p=1$, then we define the type $\tau(\mathbf{t})$ of $\mathbf{t}$ to be the empty sequence.

\subsubsection{ \ } \label{sss5.1.2}

If $p\meg 2$, then we define $\tau(\mathbf{t})$ as follows. Let $R=(r_{ij})\in A^{n\times p}$ denote the $n\times p$
matrix whose $(i,j)$-th entry $r_{ij}$ is the $i$-th coordinate of $t_j$. (More precisely, writing $t_j=(t_{1,j},\dots,t_{n,j})$
for every $j\in [p]$, we have $r_{ij}=t_{i,j}$.) Next, let $E$ denote the matrix which is obtained by first erasing all rows of $R$
with constant entries, and then shrinking all consecutive appearances of identical rows to single rows; note that $E$ is nonempty
since $p\meg 2$. Let $m$ denote the numbers of rows of $E$, and let $s_1,\dots,s_p$ denote its columns (in particular, we have that
$s_j\in A^m$ for every $j\in [p]$). We define the type $\tau(\mathbf{t})$ of $\mathbf{t}$ by the rule
\begin{equation} \label{e5.1}
\tau(\mathbf{t})=(s_1,\dots,s_p)
\end{equation}
and we call the positive integer $m$ the \textit{dimension} of $\tau(\mathbf{t})$. (Thus, $\tau(\mathbf{t})$
is a $p$-tuple of distinct elements of $A^m$.)
\begin{examp} \label{ex5.1}
Let $A=[4]$, $n=5$, $p=5$, and
\[ \mathbf{t}=\big((2,1,3,2,3),(3,1,4,2,4),(4,1,3,2,3), (1,1,4,2,4), (4,1,2,2,2)\big). \]
Then we have
\[ R= \begin{bmatrix}
      2 & 3 & 4 & 1 & 4 \\
      1 & 1 & 1 & 1 & 1 \\
      3 & 4 & 3 & 4 & 2 \\
	  2 & 2 & 2 & 2 & 2 \\
	  3 & 4 & 3 & 4 & 2 \end{bmatrix}
\ \text{ and } \
E= \begin{bmatrix}
      2 & 3 & 4 & 1 & 4 \\
      3 & 4 & 3 & 4 & 2 \end{bmatrix} \]
and, consequently, $m=2$ and $\tau(\mathbf{t})=\big((2,3),(3,4),(4,3),(1,4),(4,2)\big)$.
\end{examp}
\begin{examp} \label{ex5.2}
Let $A$ be a finite set with $|A|\meg 2$, let $\Gamma\subseteq A$ be nonempty, set $p\coloneqq |\Gamma|$,
and let $(\gamma_1,\dots,\gamma_p)$ be an enumeration of the set $\Gamma$. Also let $n\meg 1$ be an integer.
Then for every variable word $v$ over $A$ of length~$n$ we have
$\tau\big((v(\gamma_1),\dots,v(\gamma_p))\big)=(\gamma_1,\dots,\gamma_p)$.
\end{examp}
We isolate, for future use, two basic properties of types which are both straightforward consequences of the definition.
The first property shows that the type is an isomorphic invariant.
\begin{fact} \label{f5.3}
Let $A$ be a finite set with $|A|\meg 2$, let $n,p$ be positive integers such that $2\mik p\mik |A|^n$,
and let $(t_1,\dots, t_p)$ be a nonempty tuple of distinct elements of $A^n$. Then for every
$n\text{-dimensional}$ combinatorial space $V$ of $A^{<\nn}$ we have
\[ \tau\big((t_1,\dots,t_p)\big)=\tau\big( (\mathrm{I}_V(t_1),\dots,\mathrm{I}_V(t_p))\big) \]
where $\mathrm{I}_V\colon A^n\to V$ denotes the canonical isomorphism associated with $V$.
\end{fact}
The second property is the permutation invariance of types.
\begin{fact} \label{f5.4}
Let $A, n$ and $p$ be as in Fact \emph{\ref{f5.3}}. Let $(t_1,\dots, t_p)$ be a nonempty tuple of distinct elements
of $A^n$ and write $\tau\big((t_1,\dots, t_p)\big)=(s_1,\dots,s_p)$. Then for every permutation $\pi\in S_p$ we have
$\tau\big( (t_{\pi(1)},\dots, t_{\pi(p)}) \big)=( s_{\pi(1)},\dots,s_{\pi(p)})$.
\end{fact}

\subsection{The type of a nonempty finite set} \label{ss5.2}

Let $A$ be a finite set with $|A|\meg 2$, and let $n$ be a positive integer.
Let $G\subseteq A^n$ be nonempty. Set $p\coloneqq |G|$ and fix an enumeration $(t_1,\dots,t_p)$ of $G$.
If $p=1$ (that is, if $G$ is a singleton), then we define the type $\tau(G)$ of $G$ to be the empty set. Otherwise,
if $p\meg 2$, then write $\tau\big((t_1,\dots,t_p)\big)=(s_1,\dots,s_p)$ and define the type $\tau(G)$ of $G$ by setting
\begin{equation} \label{e5.2}
\tau(G)=\{s_1,\dots,s_p\}.
\end{equation}
Note that, by Fact \ref{f5.4}, $\tau(G)$ is well-defined and independent of the enumeration of~$G$, and observe
that $\tau(G)$ is a subset of $A^m$ of cardinality $|G|$ where $m$ denotes the dimension of $\tau\big((t_1,\dots,t_p)\big)$.
In a slight abuse of the previous terminology, we will call this positive integer $m$ the \textit{dimension} of $\tau(G)$.
(Note that the dimension of $\tau(G)$ bounds its cardinality; specifically, we have $|\tau(G)|\mik |A|^m$.) We set
\begin{equation} \label{e5.3}
\mathrm{Type}(A)\coloneqq\{\tau(G): G \text{ is a nonempty subset of } A^n \text{ for some integer } n\meg 1\}
\end{equation}
and we call an element of $\mathrm{Type}(A)$ a \textit{type over} $A$. We also observe the following analogue
of Fact \ref{f5.3}. (As before, the proof is straightforward.)
\begin{fact} \label{f6.5}
Let $A, n$ and $V$ be as in Fact \emph{\ref{f5.3}}. Then for every nonempty $G\subseteq A^n$ we have
$\tau(G)=\tau\big(\mathrm{I}_V(G)\big)$.
\end{fact}

\subsection{Types and the Ramsey property} \label{ss5.3}

The most important property of types is that they can be used to classify all partition
regular families of subsets of discrete hypercubes. To motivate this classification, we start
by observing that there is no analogue of Ramsey's classical theorem for colorings of subsets of combinatorial spaces
of a fixed cardinality. Indeed, let $A$ be a finite set with $|A|\meg 2$, and let $d,\ell\in\nn$ with $|A|^d\meg \ell\meg 2$.
Also let $V$ be a combinatorial space of $A^{<\nn}$ of dimension at least $d+1$, and define a coloring $c$ of the set
$\{G\subseteq V: |G|=\ell\}$ as follows. Let $G\subseteq V$ with $|G|=\ell$, and set $c(G)=\tau(G)$ if the dimension of the
type of $G$ is at most $d$; otherwise set $c(G)=0$. Regardless of how large the dimension of $V$ is, using Fact \ref{f5.3}
it is easy to see that for every $(d+1)$-dimensional combinatorial subspace $W$ of $V$ the set $\{G\subseteq W:|G|=\ell\}$
is not monochromatic.

However, colorings which depend on the type are the only obstacles to the Ramsey property. Specifically, we have the
following theorem whose proof can be found in \cite[Theorem 5.5]{DK} and which relies on the Graham--Rothschild theorem \cite{GR}.
\begin{thm} \label{thm:ramtypes}
For every triple $k,m,r$ of positive integers with $k\meg2$ there exists a positive integer $N$ with the following property.
For every integer $n\meg N$, every set $A$ with $|A|=k$ and every $r$-coloring of the powerset of $A^n$ there exists
an $m$-dimensional combinatorial subspace $V$ of $A^n$ such that every pair of nonempty subsets of\, $V$ with the same type
is monochromatic. The least positive integer $N$ with this property is denoted by $\mathrm{RamSp}(k,m,r)$.

Moreover, the numbers $\mathrm{RamSp}(k,m,r)$ are upper bounded by a primitive recursive function
belonging to the class $\mathcal{E}^6$ of Grzegorczyk's hierarchy.
\end{thm}

\subsection{Stochastic processes and types: stationarity, pseudorandomness, supercorrelation, subcorrelation} \label{ss5.4}

Our next goal is to extend Definitions \ref{defn:stasionarity_lines} and \ref{defn:correlation_lines}.
We begin by generalizing the notion of stationarity.
\begin{defn} \label{stationarity6}
Let $A$ be a finite set with $|A|\meg 2$, let $n$ be a positive integer, let $\eta>0$, and let $\langle D_t: t\in A^n\rangle$
be a stochastic process in a probability space~$(\Omega,\mathcal{F},\mathbb{P})$. We say that $\langle D_t:t\in A^n\rangle$
is \emph{$\eta$-stationary} if for every pair of nonempty sets $G_1,G_2\subseteq A^n$ with $\tau(G_1) = \tau(G_2) $ we have
\begin{equation} \label{e5.4}
\Big| \mathbb{P}\Big(\bigcap_{t \in G_1} D_t \Big)-\mathbb{P}\Big(\bigcap_{t \in G_2} D_t \Big) \Big| \mik \eta.
\end{equation}
$($In particular, by Example \emph{\ref{ex5.2}}, if a process $\langle D_t:t\in A^n\rangle$ is $\eta$-stationary,
then it is also $\eta$-stationary with respect to combinatorial lines.$)$
\end{defn}
The following fact, which extends Fact~\ref{f3.1}, is an immediate consequence of Theorem~\ref{thm:ramtypes}.
\begin{fact} \label{fact:achive_stationarity}
Let $k\meg 2$ be an integer, let $A$ be a set with $|A|=k$, let $0<\eta\mik 1$, and let $n,m$ be positive integers such that
\begin{equation} \label{e5.5}
n\meg \mathrm{RamSp}\big( k,m,\left\lceil 1/\eta\right\rceil^{2^k-1}\big).
\end{equation}
Then for every stochastic process $\langle D_t:t\in A^n\rangle$ in a probability space $(\Omega,\mathcal{F},\mathbb{P})$
there exists an $m$-dimensional combinatorial subspace $V$ of $A^n$ such that the process $\langle D_{\mathrm{I}_V(s)}:s\in A^m\rangle$
$($that is, the restriction of $\langle D_t:t\in A^n\rangle$ to $V$$)$ is $\eta$-stationary.
\end{fact}
We also have the following analogue of Lemma \ref{lem:stasionarity_lines} whose proof is identical to that of
Lemma \ref{lem:stasionarity_lines}.
\begin{lem} \label{lem:stasionarity_types}
Let $A, n, \eta, \langle D_t:t \in A^n\rangle$ be as in Definition \emph{\ref{stationarity6}}. Then the following hold.
\begin{enumerate}
\item[(i)] For every $t_1, t_2\in A^n$ we have $|\mathbb{P}(D_{t_1}) - \mathbb{P}(D_{t_1})| \mik \eta$. Thus,
for~every~$t\in A^n$ we have $|\mathbb{P}(D_t)-\ee|\mik\eta$ where $\ee\coloneqq\max\big\{\max\{\mathbb{P}(D_t):t\in A^n\},\eta\big\}>0$.
\item[(ii)] Let $m\in [n]$, and let $\tau\in\mathrm{Type}(A)$ be a type over $A$ of dimension $m$ with $|\tau|\meg 2$.
Then for every $Q\subseteq\tau$ and every pair $V_1,V_2$ of\, $m$-dimensional combinatorial subspaces of $A^n$ we have
\[ \Big| \mathbb{P} \Big( \bigcap_{t \in \mathrm{I}_{V_1}(Q)} D_t\, \cap
\bigcap_{t \in \mathrm{I}_{V_1}(\tau\setminus Q)} D_t^\complement \Big)- \mathbb{P}\Big( \bigcap_{t \in \mathrm{I}_{V_2}(Q)} D_t \,
\cap \bigcap_{t \in \mathrm{I}_{V_2}(\tau\setminus Q)} D_t^\complement \Big) \Big| \mik 2^{|Q|}\eta. \]
\end{enumerate}
\end{lem}
We proceed by generalizing the notions of pseudorandomness, supercorrelation and subcorrelation introduced in
Definition \ref{defn:correlation_lines}.
\begin{defn} \label{d5.10}
Let $A$ be a finite set with $|A|\meg2$, let $n$ be a positive integer, let $0<\eta,\ee\mik 1$, and let
$\langle D_t:t \in A^n\rangle$ be an $\eta$-stationary process in a probability space $(\Omega,\mathcal{F},\mathbb{P})$
such that $|\mathbb{P}(D_t)-\ee|\mik\eta$ for every $t\in A^n$.  Also let $\tau\in\mathrm{Type}(A)$ be a type over $A$ of
dimension at most $n$, and let $\theta\meg 0$.
\begin{enumerate}
\item[(1)] \emph{(Pseudorandomness)} We say that\, $\langle D_t:t \in A^n\rangle$ is $(\tau, \theta)$-\emph{pseudorandom}
provided that $\big|\mathbb{P}\big(\bigcap_{ t \in G } D_t\big) -\ee^{|G|}\big|\mik \theta$ for every $G\subseteq A^n$ with $\tau(G) = \tau$.
\item[(2)] \emph{(Supercorrelation)} We say that\, $\langle D_t:t \in A^n\rangle$ is $(\tau, \theta)$-\emph{supercorrelated}
provided that $\mathbb{P}\big(\bigcap_{ t \in G } D_t\big)\meg \ee^{|G|} + \theta$ for every $G\subseteq A^n$ with $\tau(G) = \tau$.
\item[(3)] \emph{(Subcorrelation)} We say that\, $\langle D_t:t \in A^n\rangle$ is $(\tau, \theta)$-\emph{subcorrelated}
provided that $\mathbb{P}\big(\bigcap_{ t \in G } D_t\big)\mik \ee^{|G|} - \theta$ for every $G\subseteq A^n$ with $\tau(G) = \tau$.
\end{enumerate}
\end{defn}
We close this section with the following analogue of Fact \ref{pr:correlations}. (Its simple proof is left to the
interested reader.)
\begin{fact} \label{pr:correlations_1_sep}
Let $A, n, \eta,\ee, \langle D_t:t \in A^n\rangle$ be as in Definition \emph{\ref{d5.10}},
let $\tau\in\mathrm{Type}(A)$ be a type over $A$ of dimension at most $n$, and let $\theta\meg \eta$.
Then one of the following holds true.
\begin{enumerate}
\item[(i)] The process $\langle D_t:t \in A^n\rangle$ is $( \tau,\theta)$-pseudorandom.
\item[(ii)] The process $\langle D_t:t \in A^n\rangle$ is $( \tau, \theta - \eta )$-supercorrelated.
\item[(iii)] The process $\langle D_t:t \in A^n\rangle$ is $( \tau, \theta - \eta )$-subcorrelated.
\end{enumerate}
\end{fact}

%--------------------The separation index---------------------%

\section{The separation index} \label{s6}

\numberwithin{equation}{section}

This section, like Section \ref{s5}, also contains preparatory material which is needed for the analysis of arbitrary correlations
of stationary stochastic processes. Our aim is to define another isomorphic invariant of nonempty subsets of discrete
hypercubes---the \textit{separation index}---which is coarser than the type, and measures how ``well-distributed" a subset is.
Specifically, we have the following definition.
\begin{defn} \label{d6.1}
Let $A$ be a finite set with $|A|\meg 2$, and let $n$ be a positive integer.
\begin{enumerate}
\item[(1)] Let $\mathbf{t}=(t_1,\dots,t_p)$ be a nonempty tuple of distinct elements of $A^n$, and let $\ell$ be a positive integer.
We say that $\mathbf{t}$ is \emph{$\ell$-separated} if for every $j\in [p]$ with $j\meg 2$ there exists $I\subseteq [n]$
$($depending, possibly, on $j$$)$ with $|I|=\ell$ and satisfying the following property: for every $q\in \{1,\dots,j-1\}$ there exists
$i\in I$ such that $t_j(i)\neq t_q(i)$. $($Namely, the $i$-th coordinate $t_j(i)$ of\, $t_j$ is different from the $i$-th
coordinate $t_q(i)$ of\, $t_q$.$)$ We define the \emph{separation index} $\mathrm{s}(\mathbf{t})$ of\, $\mathbf{t}$ to be
the least positive integer $\ell$ such that $\mathbf{t}$ is $\ell$-separated.
\item[(2)] Let $G\subseteq A^n$ be nonempty, and set $p\coloneqq |G|$. We define the \emph{separation index} $\mathrm{s}(G)$
of\, $G$ by the rule
\begin{equation} \label{e6.1}
\mathrm{s}(G)\coloneqq \min\{\mathrm{s}(\mathbf{t}): \mathbf{t}=(t_1,\dots,t_p) \text{ is an enumeration of\, } G\},
\end{equation}
and we say that $G$ is \emph{$\ell$-separated} if\, $\mathrm{s}(G)=\ell$.
\end{enumerate}
\end{defn}
\begin{rem} \label{r6.2}
Note that the separation index of a nonempty finite set may be strictly smaller than the separation index of one of its
enumerations. Consider, for instance, the set $G=\big\{(0,0),(1,0),(0,1)\big\}\subseteq \{0,1\}^2$ and
its enumeration $\mathbf{t}=\big((1,0),(0,1),(0,0)\big)$. Then we have $\mathrm{s}(\mathbf{t})=2$,
but $\mathrm{s}(G)=1$ as witnessed by the tuple $\mathbf{s}=\big((0,0),(1,0),(0,1)\big)$.
\end{rem}
In the following fact we state two basic properties of the separation index which were mentioned above, namely that it is preserved under
canonical isomorphisms and that it is coarser than the type. The proof follows from the relevant definitions and is left to the reader.
\begin{fact} \label{f6.3}
Let $A$ be a finite set with $|A|\meg 2$, let $n$ be a positive integer, and let $V$ be an $n$-dimensional combinatorial space
of $A^{<\nn}$. If\, $(t_1,\dots,t_p)$ is a nonempty tuple of distinct elements of $A^n$, then
$\mathrm{s}\big((t_1,\dots,t_p)\big)=\mathrm{s}\big((\mathrm{I}_V(t_1),\dots,\mathrm{I}_V(t_p))\big)$.
Respectively, if\, $G\subseteq A^n$ is nonempty, then $\mathrm{s}(G)=\mathrm{s}\big(\mathrm{I}_V(G)\big)$;
consequently, if $H\subseteq A^l$ for some positive integer $l$ with $\tau(H)=\tau(G)$,  then
$\mathrm{s}(H)=\mathrm{s}(G)$.
\end{fact}
We proceed by determining the separation index of some concrete examples of sets which are important from a combinatorial perspective.
\begin{examp}[Combinatorial lines] \label{ex6.4}
Let $A$ and $n$ be as in Definition \ref{d6.1}. Let $\Gamma\subseteq A$ be nonempty, and set $p\coloneqq |\Gamma|$.
Also let $v$ be a variable word over $A$ of length $n$. Then, by Fact \ref{f6.3}, for every enumeration $(\gamma_1,\dots,\gamma_p)$
of $\Gamma$ we have
\[ \mathrm{s}\big((v(\gamma_1),\dots,v(\gamma_p))\big)=\mathrm{s}\big((\gamma_1,\dots,\gamma_p)\big)=1. \]
In particular, every combinatorial line $L$ of $A^n$ is $1$-separated.
\end{examp}
\begin{examp}[Shelah lines] \label{ex6.5}
As above, let $A$ be a finite set with $|A|\meg 2$. For every $\alpha\in A$ and every positive integer $m$ let
$\alpha^m=(\alpha,\dots,\alpha)$ denote the sequence of length $m$ taking the constant value $\alpha$; also let
$\alpha^0$ denote the empty sequence.

Now let $n$ be a positive integer, let $\alpha,\beta\in A$ with $\alpha\neq\beta$, and define the
\textit{Shelah line\footnote{These sets play a crucial role in Shelah's proof of the Hales--Jewett theorem.}
with parameters $\alpha,\beta$} by rule
\begin{equation} \label{e6.2}
\mathrm{S}=\big\{\alpha^{n-m\con}\beta^m: m\in\{0,\dots,n\}\big\}\subseteq A^n.
\end{equation}
Clearly, we have $|\mathrm{S}|=n+1$, and it is easy to see that the set $\mathrm{S}$ is $1$-separated.
\end{examp}
Example \ref{ex6.5} implies, in particular, that there exist $1$-separated sets of arbitrarily large cardinality.
More generally, we have the following lemma.
\begin{lem}[Random tuples of small size are $1$-separated] \label{l6.6}
Let $k,n,p$ be positive integers with $k\meg 2$ and $2\mik p\mik k^n$. Let $A$ be a set with $|A|=k$, and let $\mathbb{P}$
denote the uniform probability measure on $(A^n)^p$. $($That is, $(A^n)^p$ is the Cartesian product of $p$ many copies of $A^n$.$)$
Then we have
\begin{equation} \label{e6.3}
\mathbb{P}(\mathbf{t} \text{ is $1$-separated})\meg 1-p e^{-n(\frac{k-1}{k})^p}.
\end{equation}
In particular, if $p\mik \log(n)$, then $\mathbb{P}(\mathbf{t} \text{ is $1$-separated})=1-o_{n\to\infty;k}(1)$.
\end{lem}
\begin{proof}
Set $S\coloneqq\{\mathbf{t}\in (A^n)^p: \mathbf{t} \text{ is $1$-separated}\}$, and let $S^\complement$ denote the complement of $S$.
For every $i\in[n]$ and every $j\in[p]$ with $j\meg2$ the set of all $\mathbf{t}=(t_1,\dots,t_p)\in (A^n)^p$ such that
$t_j(i)\notin\{t_1(i),\dots, t_{j-1}(i)\}$ has probability $\frac{k(k-1)^{j-1}}{k^j}\meg \big(\frac{k-1}{k}\big)^p$.
(Here,  $t_q(i)$ denotes the $i$-th coordinate of $t_q$ for every $q\in [j]$.) Therefore, for every $j\in [p]$ with $j\meg2$
the set of all $\mathbf{t}=(t_1,\dots,t_p)\in (A^n)^p$ such that $t_j$ fails to satisfy the condition of being $1\text{-separated}$
has probability at most~$\big(1-\big(\frac{k-1}{k}\big)^p\big)^n$. Using the fact that $(1-\frac{r}{n})^n\mik e^{-r}$
for every $r>0$ and every positive integer~$n$, we thus have
\begin{equation} \label{e6.4}
\mathbb{P}(S^\complement) \mik p\Big(1-\Big(\frac{k-1}{k}\Big)^p\Big)^n \mik p e^{-n(\frac{k-1}{k})^p}
\end{equation}
which is equivalent to \eqref{e6.3}.

Next assume that $p\mik \log(n)$. Since the function  $f(x)=x e^{-nr^x}$ is increasing for every $r\in (0,1)$ and
every positive integer $n$, by \eqref{e6.4}, we obtain that
%\[ \mathbb{P}(S^\complement) \mik \log(n) e^{-n(\frac{k-1}{k})^{\log(n)}}
%  =\log(n) e^{-n(\frac{k-1}{k})^{\frac{\log_\frac{k-1}{k}(n)}{\log_\frac{k-1}{k}(e)}}}
%  =\log(n) e^{-n^{1+\frac{1}{\log_\frac{k-1}{k}(e)}}}  =\log(n) e^{-n^{1+\log(\frac{k-1}{k})}}.\]
\[ \mathbb{P}(S^\complement) \mik \log(n)\, e^{-n(\frac{k-1}{k})^{\log(n)}} = \log(n)\, e^{-n^{1+\log(\frac{k-1}{k})}}. \]
Therefore, $\mathbb{P}(S)=1-o_{n\to\infty;k}(1)$ as desired.
\end{proof}
The last example in this section provides us with a representative example of an $n\text{-separated}$ set.
\begin{examp}[Combinatorial subspaces] \label{ex6.7}
Let $A$ and $n$ be as in Definition \ref{d6.1}, and notice that for every nonempty $G\subseteq A^n$ we have $\mathrm{s}(G)\mik n$.
On the other hand, it is easy to verify that $\mathrm{s}(A^n)=n$. Using this observation and Fact \ref{f6.3}, we see that
every $n$-dimensional combinatorial space of $A^{<\nn}$ is $n$-separated.
\end{examp}

%------------Correlations over $1$-separated sets-------------%

\section{Correlations over $1$-separated sets} \label{s7}

\numberwithin{equation}{section}

\subsection{The main result} \label{s7.1}

We begin by introducing the analogue of insensitivity for processes indexed by combinatorial spaces.
\begin{defn} \label{d7.1}
Let $A, n, \alpha, \beta$ be as in Definition \emph{\ref{insensitivity-intro}}, let $V$ be an $n\text{-dimensional}$
combinatorial space of $A^{<\nn}$, and let $\mathrm{I}_V\colon A^n\to V$ denote the canonical isomorphism associated
with $V$. We say that a stochastic process $\langle D_t:t\in V\rangle$ in a probability space $(\Omega,\mathcal{F},\mathbb{P})$
is \emph{$(\alpha,\beta)$-insensitive~in~$V$} if $\langle D_{\mathrm{I}_V(s)}:s\in A^n\rangle$ is $(\alpha,\beta)$-insensitive
in the sense of Definition~\emph{\ref{insensitivity-intro}}. $($That is, if\, $D_{\mathrm{I}_V(s)}=D_{\mathrm{I}_V(t)}$
for every $s,t\in A^n$ which are $(\alpha,\beta)$-equivalent.$)$
\end{defn}
The main result of this section is the following extension of Theorem \ref{thm:dich_lines} which concerns correlations
of stationary processes over $1$-separated sets. (We recall that the notion of stationarity in this more general context
is given in Definition \ref{stationarity6}.)
\begin{thm} \label{t7.2}
Let $k,\kappa,m$ be positive integers such that $k,\kappa \meg 2$ and $\kappa\mik k^m$, and let $\ee,\sigma,\eta>0$ with
\begin{equation} \label{e7.1}
\ee \mik 1-\frac{1}{2\kappa}, \ \  \sigma \mik \frac{\ee^{\kappa- 1}}{2\kappa} \ \text{ and } \
\eta\mik \frac{\sigma}{4^{\kappa- 1}}.
\end{equation}
Also let $A$ be a set with $|A|=k$, let $n>m$ be an integer, and let $\langle D_t: t\in A^n\rangle$
be an $\eta$-stationary process in a probability space $(\Omega, \mathcal{F}, \mathbb{P})$
such that $|\mathbb{P}(D_t)-\ee|\mik\eta $ for every $t\in A^n$. Then, either
\begin{enumerate}
\item [(i)] for every nonempty $1$-separated $G\subseteq A^n$ with cardinality at most $\kappa$ and whose type $\tau(G)$ has dimension
at most $m$ we have
\begin{equation} \label{e7.2}
\Big|\mathbb{P}\Big(\bigcap_{t \in G} D_t\Big)-\ee^{|G|}\Big|\mik \sigma,
\end{equation}
\item[(ii)] or $\langle D_t: t\in A^n\rangle$ correlates with a ``structured" stochastic process when restricted to a large subspace;
precisely, there exist a combinatorial subspace $V$ of $A^n$ with $\mathrm{dim}(V)\meg n-m$, a nonempty subset $\Gamma$ of $A$,
$\beta\in A \setminus \Gamma$ and a stochastic process $\langle S_t: t \in V\rangle$ in $( \Omega, \mathcal{F}, \mathbb{P})$
with the following properties.
\begin{enumerate}
\item[(a)] For every $t\in V$ we have $S_t=\bigcap_{\alpha\in\Gamma} E^\alpha_t$ where for every $\alpha\in \Gamma$
the process $\langle E_t^\alpha: t\in V\rangle$ is $(\alpha, \beta)$-insensitive in $V$.
\item[(b)] For every $t\in V$ we have
\begin{equation} \label{e7.3}
\mathbb{P}(S_t)\meg\frac{\ee^{\kappa- 1}}{4\kappa} \ \text{ and } \ \,
\mathbb{P}(D_t\, |\, S_t) \meg \ee + \frac{\sigma}{4^{\kappa-1}}.
\end{equation}
\end{enumerate}
\end{enumerate}
\end{thm}
Theorem \ref{t7.2} shows that stationary processes which exhibit non-independent behavior over $1$-separated sets
are essentially characterized---in the strong quantitative sense described in \eqref{e7.3}---by their correlation
with insensitive processes. Note, however, that in contrast to Theorem \ref{thm:dich_lines}, this correlation
is ``local" in nature, that is, we need to pass to a subspace in order to verify it. We present an example
in Subsection \ref{ss7.2} which elucidates the necessity of this restriction.

The proof of Theorem \ref{t7.2} is given in Subsection \ref{ss7.5}. It relies on the following analogue of
Proposition \ref{thm:step_cor_lines_simple} whose proof is given in Subsection \ref{ss7.4}. (The concepts
of pseudorandomness, supercorrelation and subcorrelation which appear below are introduced in Definition \ref{d5.10}.)
\begin{prop} \label{p7.3}
Let $A$ be a finite set with $|A|\meg2$, let $n, p$ be positive integers with $p+1\mik |A|^n$, let $0<\eta,\ee\mik 1$, and let
$\langle D_t:t \in A^n\rangle$ be an $\eta$-stationary process in a probability space $(\Omega,\mathcal{F},\mathbb{P})$
such that $|\mathbb{P}(D_t)-\ee|\mik\eta$ for every $t\in A^n$. Let~$\mathbf{t}=(t_1,\dots, t_{p+1})$ be an
$1$-separated tuple consisting of distinct elements of~$A^n$, set $G\coloneqq \{t_1,\dots,t_{p+1}\}$ and $H\coloneqq\{t_1,\dots,t_p\}$,
and let $d$ denote the dimension of~$\tau(G)$. Finally, let $0<\theta,\sigma\mik 1$, and assume that the process $\langle D_t:t \in A^n\rangle$
is $(\tau(H),\theta)\text{-pseudorandom}$. Then there exist an $(n-d)$-dimensional combinatorial subspace $V$ of $A^n$,
a nonempty subset $\Gamma$ of $A$, $\beta\in A\setminus\Gamma$ and a process $\langle S_t:t \in V\rangle$
in $(\Omega,\mathcal{F},\mathbb{P})$ with the following properties.
\begin{enumerate}
\item [(i)] For every $t\in V$ we have $S_t=\bigcap_{\alpha \in \Gamma}  E^\alpha_t$ where
for every $\alpha\in\Gamma$ the process $\langle E_t^\alpha:t \in V\rangle$ is $(\alpha, \beta)$-insensitive in $V$.
\item[(ii)] For every $t\in V$ we have $|\mathbb{P}(S_t)-\ee^p|\mik \theta$.
\item[(iii)] If\, $\langle D_t:t \in A^n\rangle$ is $(\tau(G),\sigma)$-supercorrelated, then for every $t\in V$ we have
\begin{equation} \label{e7.4}
\mathbb{P}(D_t\, |\, S_t)\meg\ee\, \Big(1+\frac{\sigma\ee^{-1}-\theta}{\ee^p+\theta}\Big).
\end{equation}
\item [(iv)] If\, $\langle D_t:t \in A^n\rangle$ is $(\tau(G),\sigma)$-subcorrelated, then for every $t\in V$ we have
\begin{equation} \label{e7.5}
\mathbb{P}(D_t\, |\, S_t )\mik\ee\, \Big(1-\frac{\sigma\ee^{-1}-\theta}{\ee^p-\theta}\Big).
\end{equation}
\end{enumerate}
\end{prop}

\subsection{Correlations over $1$-separated sets: example} \label{ss7.2}

We are about to present an
example of a process which exhibits non-independent behavior when we look at its correlations over $1$-separated sets
whose type is rather simple, but not quite similar to that of combinatorial lines.
As in Example \ref{examp-intro}, for concreteness we will work with the set $A=\{1,2,3\}$ and the $1$-separated type
\begin{equation} \label{e7.6}
\tau=\big\{(1,2,1),(2,1,2),(2,2,3)\big\}\in \mathrm{Type}\big(\{1,2,3\}\big).
\end{equation}
Let $\ee>0$, let $n\meg 5$ be an integer, and let
\[ \big\langle E_{y^\smallfrown s}:y\in\{1,2,3\}^3\setminus \{(2,2,3)\}\,\text{ and }\, s\in\{1,2\}^{n-2}\big\rangle \]
be a family of independent events in a probability space $(\Omega,\mathcal{F},\mathbb{P})$ with equal probability~$\sqrt{\ee}$.
We define $\langle D_t:t\in \{1,2,3\}^n\rangle$ by setting for every $z\in\{1,2,3\}^{n-3}$
\begin{enumerate}
\item[(a)] $D_{(2,2,3)^{\con}z}\coloneqq E_{(1,2,1)^{\con}{z^{3\to 1}}^{\con}1}\cap
E_{(2,1,2)^{\con}{z^{3\to 2}}^{\con}2}$, and
\item[(b)] $D_{y^\smallfrown z}\coloneqq E_{y^{\con}{z^{3\to 1}}^{\con}1}\cap E_{y^{\con}{z^{3\to 2}}^{\con}2}$
if $y\in\{1,2,3\}^3\setminus\{(2,2,3)\}$.
\end{enumerate}
Note the difference between the definition in (a) and the definition in Example \ref{examp-intro}: given $t\in A^n$,
first we change a short initial segment of $t$ and then we ``project" the rest of the sequence. This maneuver will
be generalized in the next subsection.

The analysis of the correlations of $\langle D_t:t\in \{1,2,3\}^n\rangle$ is fairly straightforward. Specifically,
notice that for every $t\in \{1,2,3\}^n$ we have $\mathbb{P}(D_t)=\ee$. Moreover, for every $z\in\{1,2,3\}^{n-3}$ set
\[ G_z\coloneqq \big\{ (1,2,1)^{\con}(z^{3\to 1}), (2,1,2)^{\con}(z^{3\to 2}),
(2,2,3)^{\con}z\big\}\subseteq \{1,2,3\}^n \]
and observe that $\tau(G_z)=\tau$ and $\mathbb{P}\big(\bigcap_{t\in G_z} D_t\big)=\ee^2$ which deviates, of course,
from the expected value $\ee^3$.

Finally, note that $\langle D_t:t\in \{1,2,3\}^n\rangle$ cannot be written as the intersection of insensitive processes,
but only barely so. Indeed, set
\begin{equation} \label{e7.7}
V\coloneqq \big\{ (2,2,3)^\smallfrown z:z\in\{1,2,3\}^{n-3}\big\}
\end{equation}
and observe that $V$ is an $(n-3)$-dimensional combinatorial subspace of $\{1,2,3\}^n$. Clearly, by (a) above,
the restriction of $\langle D_t:t\in \{1,2,3\}^n\rangle$ to $V$ is the intersection of two
processes which are $(1,3)$- and $(2,3)$-insensitive in $V$ respectively.

\subsection{Definitions/Notation} \label{ss7.3}

Let $A,n$ and $p$ be as in Proposition \ref{p7.3}. Let
\begin{equation} \label{e7.8}
\mathbf{t}=(t_1,\dots,t_{p+1})
\end{equation}
be an $1$-separated tuple consisting of distinct elements of $A^n$, let $\tau=\tau(\mathbf{t})$ denote its type, and let $d$ denote
the dimension of $\tau$. We will define
\begin{enumerate}
\item[$\bullet$] an $(n-d)$-dimensional combinatorial subspace $V$ of $A^n$,
\item[$\bullet$] a nonempty subset $\Gamma$ of $A$,
\item[$\bullet$] $\beta\in A\setminus\Gamma$,
\item[$\bullet$] an integer $\iota\in [d]$, and
\item[$\bullet$] for every $j\in [p]$ a map $T_j\colon V\to A^n$.
\end{enumerate}
These data will be used in the proofs of Theorem \ref{t7.2} and Proposition \ref{p7.3}---in fact, they constitute the combinatorial heart
of the argument. We also note that $V,\Gamma,\beta,\iota$ and $\langle T_j:j\in [p]\rangle$ will essentially depend upon the
type $\tau$ of $\mathbf{t}$ and not on the tuple~$\mathbf{t}$ itself; however, it is technically easier to work with $\mathbf{t}$.

\subsubsection{Defining $\iota,\beta$ and $\Gamma$, and splitting the type $\tau$} \label{sss7.3.1}

We write $\tau=(s_1,\dots,s_{p+1})$ where $s_j=\big(s_j(1),\dots,s_j(d)\big)\in A^d$ for every $j\in [p+1]$.
By Fact \ref{f6.3} and our assumption that $\mathbf{t}$ is $1$-separated, we see that $\tau$ is also $1$-separated.
Taking into account this remark, we define
\begin{equation} \label{e7.9}
\iota\coloneqq \min\big\{ i\in [d]: s_{p+1}(i)\neq s_j(i) \text{ for every } j\in [p]\big\}
\end{equation}
and
\begin{equation} \label{e7.10}
\beta\coloneqq s_{p+1}(\iota),  \ \ \beta_j\coloneqq s_j(\iota) \text{ for every } j\in [p], \ \text{ and } \
\Gamma\coloneqq\{\beta_1,\dots,\beta_p\}.
\end{equation}
(In particular, we have $\beta\notin \Gamma$; also note that $|\Gamma|\mik p$ since the elements $\beta_1,\dots,\beta_p$ are
not necessarily distinct.) Moreover, for every $j\in [p+1]$ set
\begin{equation} \label{e7.11}
x_j=\big(s_j(1),\dots,s_j(\iota)\big) \ \text{ and } \ y_j=\big(s_j(\iota+1),\dots,s_j(n)\big)
\end{equation}
with the convention that $y_j$ is the empty sequence if $\iota=d$; note that $s_j={x_j}^{\con}y_j$.

\subsubsection{Defining $V$ and the maps $\langle T_j:j\in [p]\rangle$} \label{sss7.3.2}

Next, set
\begin{equation} \label{e7.12}
V\coloneqq \{ {x_{p+1}}^{\!\con}z^{\con}y_{p+1}:z\in A^{n-d}\}
\end{equation}
and observe that $V$ is an $(n-d)$-dimensional combinatorial subspace\footnote{Notice that the subspace $V$
is of very special form; in particular, the canonical isomorphism associated with $V$ is the map
$A^{n-d}\ni z\mapsto {x_{p+1}}^{\!\con}z^{\con}y_{p+1}\in V$.} of $A^n$.
Finally, for every $j\in [p]$ we define $T_j\colon V\to A^n$ by the rule
\begin{equation} \label{e7.13}
T_j({x_{p+1}}^{\!\con}z^{\con}y_{p+1})= {x_j}^{\con}(z^{\beta\to\beta_j})^{\con}y_j.
\end{equation}

%------------------------FIGURE-------------------------%
\begin{figure}[htb] \label{figure1}
\centering \includegraphics[width=.40\textwidth]{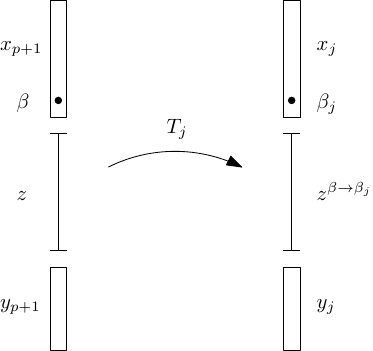}
\caption{The map $T_j$ acting on $V$.}
\end{figure}
%------------------------FIGURE-------------------------%

\subsubsection{Basic properties} \label{sss7.3.3}

We close this subsection by observing the following two elementary
(though important) properties of the previous constructions.
\begin{fact} \label{f7.4}
Let $\mathbf{t}, V,\beta,\beta_1,\dots,\beta_p$ and $\langle T_j:j\in [p]\rangle$ be as above.
\begin{enumerate}
\item[(i)] For every $t\in V$ and every $1\mik i_1<\cdots < i_q\mik p$ we have
\begin{equation} \label{e7.14}
\tau\big( (T_{i_1}(t),\dots,T_{i_q}(t),t)\big)=\tau\big( (t_{i_1},\dots,t_{i_q},t_{p+1})\big)
\end{equation}
and
\begin{equation} \label{e7.15}
\tau\big( (T_{i_1}(t),\dots,T_{i_q}(t)) \big)=\tau\big( (t_{i_1},\dots,t_{i_q})\big).
\end{equation}
\item[(ii)] Let $\langle D_t:t\in A^n\rangle$ be a stochastic process in a probability space
$(\Omega,\Sigma,\mu)$. Then for every $j\in[p]$ the process $\langle D_{T_j(t)}\!:\!t\in V\rangle$ is $(\beta_j,\beta)$-insensitive in~$V$.
\end{enumerate}
\end{fact}

\subsection{Proof of Proposition \ref{p7.3}} \label{ss7.4}

Let $V,\beta,\beta_1,\dots,\beta_p,\Gamma$ and $\langle T_j:j\in [p]\rangle$
be the data obtained in Subsection \ref{ss7.3} for the $1$-separated tuple $\mathbf{t}=(t_1,\dots,t_{p+1})$.

For every $t\in V$ define $S_t=\bigcap_{j=1}^p D_{T_j(t)}$ and notice that, by part (ii) of Fact~\ref{f7.4}, the
process $\langle S_t:t \in V \rangle$ satisfies part (i) of the theorem. On the other hand,
since the process $\langle D_t:t \in A^n \rangle $ is $(\tau(H),\theta)$-pseudorandom,
by \eqref{e7.15}, for every $t\in V$~we~have
\begin{equation} \label{e7.16}
|\mathbb{P}(S_t)-\ee^p|\mik \theta;
\end{equation}
that is, $\langle S_t:t \in V \rangle$ satisfies part (ii) of the theorem.

For part (iii) assume that $\langle D_t:t \in A^n \rangle$ is $(\tau(G),\sigma)$-supercorrelated,
and let $t\in V$ be arbitrary. By \eqref{e7.14} and the supercorrelation assumption, we have
$\mathbb{P}(D_t\cap S_t)\meg\ee^p+\sigma$. On the other hand, by \eqref{e7.16}, we see that
$\mathbb{P}(S_t)\mik \ee^p+\theta$. Therefore,
$\mathbb{P}(D_t\, |\, S_t)\meg\ee\big(1+\frac{\sigma\ee^{-1}-\theta}{\ee^p+\theta}\big)$ as desired.

Finally, for part (iv) assume that $\langle D_t:t \in A^n \rangle$ is $(\tau(G),\sigma)$-subcorrelated,
and fix $t\in V$. Using \eqref{e7.14} once again, the subcorrelation assumption and \eqref{e7.16}, we obtain that
$\mathbb{P}(D_t\cap S_t)\mik\ee^p-\sigma$ and $\mathbb{P}(S_t)\meg\ee^p-\theta$
which implies that $\mathbb{P}(D_t\, |\, S_t)\mik\ee\big(1-\frac{\sigma\ee^{-1}-\theta}{\ee^p-\theta}\big)$.
The proof is completed.

\subsection{Proof of Theorem \ref{t7.2}} \label{ss7.5}

It is similar to the proof of Theorem \ref{thm:dich_lines} with the main new ingredients being Proposition \ref{p7.3}
and the material in Subsection \ref{ss7.3}. We shall describe in detail the necessary changes, as this proof will also
serve as a model for the proof of Theorem \ref{t8.5} in Section \ref{s8}.

Let $(\theta_p)_{p=0}^\kappa$ be the finite sequence defined in \eqref{e3.6}---that is, $\theta_0=0$, $\theta_1=\eta$, and
$\theta_p=4^{p-\kappa}\sigma$ if $p\in\{2,\dots,\kappa\}$---and recall that $(\theta_p)_{p=0}^\kappa$ is increasing.
Assume that part (i) of the theorem does not hold true, and fix an $1$-separated set $G\subseteq A^n$ of cardinality
at most $\kappa$ whose type $\tau(G)$ has dimension at most $m$ and such that: (a) the process $\langle D_t:t\in A^n\rangle$
is not $(\tau(G),\theta_{|G|})$-pseudorandom, and~(b)~$G$ has the minimal cardinal among all sets with these properties.
(Note that $|G|\meg 2$.) Let $\mathbf{t}=(t_1,\dots,t_{|G|})$ be an enumeration of $G$ such that the tuple $\mathbf{t}$
is $1$-separated, let $d$ denote the dimension of $\tau(G)$, and set $H\coloneqq\{t_1,\dots,t_{|G|-1}\}$ and $p\coloneqq|H|$;
notice that $1\mik p\mik \kappa-1$ and $1\mik d\mik m$. Also observe that for every nonempty proper subset $\Sigma$ of $G$
the process $\langle D_t:t\in A^n\rangle$ is $(\tau(\Sigma),\theta_{|\Sigma|})$-pseudorandom.

We set $\theta\coloneqq\theta_p$ and $\Theta\coloneqq\theta_{p+1}$.
Since $\langle D_t:t\in A^n\rangle$ is not $(\tau(G),\Theta)$-pseudorandom, by Fact \ref{pr:correlations_1_sep}, we see that either
\begin{enumerate}
\item[(A1)] the process $\langle D_t:t\in A^n\rangle$ is $(\tau(G),\Theta-\eta)$-supercorrelated,
\item[(A2)] or the process $\langle D_t:t\in A^n\rangle$ is $(\tau(G),\Theta-\eta)$-subcorrelated.
\end{enumerate}
If the first case holds true, then, arguing precisely as in the proof of Theorem~\ref{thm:dich_lines} and using
Proposition \ref{p7.3} instead of Proposition \ref{thm:step_cor_lines_simple}, it is easy to verify that part (ii)
of the theorem is satisfied.

So assume that the process $\langle D_t: t \in A^n\rangle$ is $(\tau(G),\Theta-\eta)$-subcorrelated, and let
$V$ and $\langle T_j:j\in [p]\rangle$ be the combinatorial space and the maps obtained in Subsection \ref{ss7.3}
for the $1$-separated tuple $\mathbf{t}$. For every $t\in V$ we set
\begin{equation} \label{e7.17}
S_t\coloneqq \Big(\bigcap_{j=1}^{p-1}D_{T_j(t)}\Big)\cap D_{T_p(t)}^\complement =
\Big(\bigcap_{j=1}^{p-1}D_{T_j(t)}\Big) \setminus \Big(\bigcap_{j=1}^pD_{T_j(t)}\Big).
\end{equation}
(Recall that, by convention, $\bigcap_{j=1}^{p-1}D_{T_j(t)}=\Omega$ if $p=1$.)
Notice that, by part (ii) of Fact \ref{f7.4}, the process $\langle S_t:t\in V\rangle$ satisfies
part (ii.a) of the theorem. Next, we set $F\coloneqq\{t_1,\dots,t_{p-1}\}$ (observe that $F$ may be empty).
Since the process $\langle D_t:t\in A^n\rangle$ is $(\tau(F),\theta_{p-1})$-pseudorandom if $p>1$ (if $p=1$, then
this is superfluous) and, in addition, $(\tau(H),\theta)$-pseudorandom, by \eqref{e7.15}, for every $t\in V$ we have
\begin{equation} \label{e7.18}
|\mathbb{P}(S_t)-\ee^{p-1}(1-\ee)|\mik \theta+\theta_{p-1}.
\end{equation}
Using \eqref{e7.14} and the fact that the process $\langle D_t:t\in A^n\rangle$ is $(\tau(G),\Theta-\eta)$-subcorrelated
and $(\tau(F\cup\{t_{p+1}\}),\theta)$-pseudorandom, for every $t\in V$ we have
\begin{equation} \label{e7.19}
\mathbb{P}(D_t\cap S_t)\meg (\ee^p-\theta)-(\ee^{p+1}-\Theta+\eta).
\end{equation}
Moreover, by \eqref{e7.1} and the definition of $(\theta_j)_{j=1}^\kappa$, we see that
$\theta+\theta_{p-1}\mik\ee^p$, $\theta\mik\Theta/4$ and $\theta_{p-1}+\eta\mik\Theta/4$.
Therefore, by \eqref{e7.18} and \eqref{e7.19}, for every $t\in V$
\[ \mathbb{P}(D_t\, |\, S_t) \meg \frac{\ee^p-\ee^{p+1}+\Theta-\theta-\eta}{\ee^{p-1}-\ee^p+\theta+\theta_{p-1}}
\meg\ee+\frac{\Theta-2\theta-\theta_{p-1}-\eta}{\ee^{p-1}-\ee^p+\theta+\theta_{p-1}} \meg
\ee+\frac{\Theta}{4}\meg\ee+\frac{\sigma}{4^{\kappa-1}}. \]
Finally, by \eqref{e7.1} and \eqref{e7.18}, we conclude that
\[ \mathbb{P}(S_t)\meg\ee^{p-1}-\ee^p-\theta-\theta_{p-1}\meg\ee^{\kappa-1}(1-\ee)-2\theta
\meg\frac{\ee^{\kappa-1}}{2\kappa}-2\theta\meg\frac{\ee^{\kappa-1}}{4\kappa} \]
for every $t\in V$. The proof is completed.

%----------Correlations over $\ell$-separated sets------------%

\section{Correlations over $\ell$-separated sets} \label{s8}

\numberwithin{equation}{section}

\subsection{Obstructions to independence: simplicial processes} \label{ss8.1}

We are about to begin our analysis of arbitrary correlations of stationary processes.
As we have noted, the main---and perhaps the most interesting---difference lies in the
fact that insensitive process are not enough to characterize non-independent behavior. Our goal in this subsection is to
discuss this phenomenon and introduce the ``structured" processes which appear in this more general context.

To this end, we need to define a ``local" version of insensitivity. To motivate this ``local" version, let $A$ be a finite
set with $|A|\meg 2$, let $n,\ell,r_1,\dots,r_\ell$ be positive integers with $n,\ell\meg 2$ and $n=r_1+\cdots+r_\ell$,
and note that we may identify the hypercube $A^n$ with the product $A^{r_1}\times\cdots\times A^{r_\ell}$ via the map
\[ A^{r_1}\times\cdots\times A^{r_\ell}\ni (t_1,\dots,t_\ell)\mapsto {t_1}^{\con}\dots^{\con}t_\ell\in A^n. \]
Having in mind this identification, we may consider subsets of $A^n$ which are insensitive only in one of the factors
$A^{r_1},\dots,A^{r_\ell}$. This is, essentially, the content of the following definition.
\begin{defn}[Local insensitivity] \label{d8.1}
Let $A$ be a finite set with $|A|\meg2$, let $n$ be a positive integer, let $\alpha, \beta \in A$ with $\alpha\neq\beta$,
and let $I\subseteq [n]$ be nonempty. Also let $(\Omega,\mathcal{F},\mathbb{P})$ be a probability space.
\begin{enumerate}
\item[(1)] Let $t,s\in A^n$ and write $t=(t_1,\dots, t_n)$ and $s=(s_1,\dots,s_n)$. We say that $t,s$ are
\emph{$(\alpha,\beta,I)$-equivalent} if for every $i\in [n]\setminus I$ we have $t_i = s_i$ and, moreover,
for every $i\in I$ and every $\gamma\in A \setminus \{\alpha, \beta\}$ we have $t_i=\gamma$ if and only if $s_i=\gamma$.
\item[(2)] We say that a process $\langle E_t:t \in A^n\rangle$ in $(\Omega,\mathcal{F},\mathbb{P})$
is \emph{$(\alpha,\beta,I)$-insensitive} if $E_t = E_s$ for every $t,s\in A^n$ which are $(\alpha, \beta,I)$-equivalent.
\item[(3)] Let $V$ be an $n$-dimensional combinatorial space of $A^{<\nn}$. We say that a  process
$\langle E_t:t \in V\rangle$ in $(\Omega,\mathcal{F},\mathbb{P})$ is \emph{$(\alpha,\beta,I)$-insensitive in $V$}
provided that $\langle E_{\mathrm{I}_V(t)}:t \in A^n\rangle$ is $(\alpha, \beta, I )$-insensitive
where $\mathrm{I}_V\colon A^n\to V$ denotes the canonical isomorphism associated with $V$.
\end{enumerate}
\end{defn}
We proceed with the following example which shows the need to extend the notion of a ``structured" process.
\begin{examp} \label{ex8.2}
As in Example \ref{examp-intro}, we will work with the set $A=\{1,2,3\}$, and we will focus on correlations over
$2$-dimensional combinatorial spaces of $\{1,2,3\}^{<\nn}$. Notice that, by Fact \ref{f6.5} and Example \ref{ex6.7},
all $2$-dimensional combinatorial spaces are $2$-separated and are of type
\[ \tau=\big\{(1,1),(1,2),(1,3),(2,1),(2,2),(2,3),(3,1),(3,2),(3,3)\big\}\in \mathrm{Type}(\{1,2,3\}). \]
Now let $n$ be an arbitrary positive integer, and fix a family
\[ \big\langle E_{t^{\con}s}: t^{\con}s\in (\{1,2\}^n\times\{1,2,3\}^n) \cup (\{1,2,3\}^n\times\{1,2\}^n)\big\rangle \]
of independent events in a probability space $(\Omega,\mathcal{F},\mathbb{P})$ with equal probability~$\ee>0$.
We define a process $\langle D_z:z\in\{1,2,3\}^{2n}\rangle$ by setting
\begin{equation} \label{e8.1}
D_{t^\con s} \coloneqq S^1_{t^{\con}s} \cap S^2_{t^{\con} s}
\end{equation}
where $S^1_{t^{\con}s}\coloneqq E_{t^{3\to1\con}s^{3\to1}} \cap E_{t^{3\to1\con}s^{3\to2}} \cap
E_{t^{3\to1\con}s} \cap E_{t^{3\to2\con}s^{3\to1}} \cap E_{t^{3\to2\con}s^{3\to2}} \cap E_{t^{3\to2\con}s}$ and
$S^2_{t^{\con}s}\coloneqq E_{t^\con s^{3\to1}} \cap E_{t^\con s^{3\to2}}$ for every $t,s\in\{1,2,3\}^n$.
Note that
\begin{equation} \label{e8.2}
\mathbb{P}(D_{t^\con s})=
\begin{cases}
\ee^8 & \text{if both $t,s$ contain $3$}, \\
\ee^3 & \text{if exactly one of $t,s$ contains $3$},\\
\ee & \text{if both $t,s$ do not contain $3$}.
\end{cases}
\end{equation}
Next, for every $t,s\in\{1,2,3\}^n$ which contain $3$ set
\begin{eqnarray*}
G_{t,s} & \coloneqq & \big\{ t^{3\to1\con}s^{3\to1}, t^{3\to1\con}s^{3\to2}, t^{3\to1\con}s, t^{3\to2\con}s^{3\to1}, \\
& & \ \ \ \ \ \ t^{3\to2\con}s^{3\to2},t^{3\to2\con}s,t^\con s^{3\to1}, t^\con s^{3\to2}, t^\con s\big\}
\end{eqnarray*}
and observe that $G_{t,s}$ is a $2$-dimensional combinatorial subspace of $\{1,2,3\}^{2n}$; also notice that
$D_{t^\con s}=\bigcap_{z\in G_{t,s}}D_z$ and, therefore, $\mathbb{P}(\bigcap_{z\in G_{t,s}}D_z)=\ee^8$ which
deviates from the expected value $\ee^{24}$. In other words, the process $\langle D_z:z\in\{1,2,3\}^{2n}\rangle$
exhibits non-independent behavior when we look at its correlations over $2$-dimensional combinatorial subspaces.

Note, however, that $\langle D_z:z\in\{1,2,3\}^{2n}\rangle$ \textit{cannot} be written as the intersection of insensitive
processes even if we restrict it to subspaces of very small dimension. (This is a consequence of the fact that
$\langle S^1_{t^{\con}s}:t,s\in\{1,2,3\}^n\rangle$ and $\langle S^2_{t^{\con}s}:t,s\in\{1,2,3\}^n\rangle$
depend non-trivially on the parameters $s$ and $t$ respectively.) Nevertheless, the process $\langle D_z:z\in\{1,2,3\}^{2n}\rangle$
is not random at all: it is obtained from $\langle S^1_{t^{\con}s}:t,s\in\{1,2,3\}^n\rangle$
and $\langle S^2_{t^{\con}s}:t,s\in\{1,2,3\}^n\rangle$ which are both the intersection of \textit{locally} insensitive processes
but for disjoint domains of insensitivity. This less restrictive form of structurability is abstracted in the following definition.
\end{examp}
\begin{defn}[Simplicial processes] \label{d8.3}
Let $A$ be a finite set with $|A|\meg2$, let $n,\ell$ be positive integers with $n\meg \ell$, let $\mathbf{r}=(r_1,\dots,r_\ell)$
be an $\ell$-tuple of positive integers such that $n=r_1+\cdots+r_\ell$, and let $I_1^\mathbf{r},\dots,I_\ell^\mathbf{r}$ denote
the unique successive intervals of $[n]$ such that $|I_l^\mathbf{r}|=r_l$ for every $l\in[\ell]$. We say that a process
$\langle S_t:t\in A^n\rangle$ in a probability space $(\Omega,\mathcal{F},\mathbb{P})$ is \emph{$(\ell,\mathbf{r})$-simplicial}
if there exist
\begin{enumerate}
\item[$\bullet$] $\beta_1,\dots,\beta_\ell\in A$,
\item[$\bullet$] for every $l\in [\ell]$ a nonempty subset $\Gamma_l$ of $A\setminus\{\beta_l\}$, and
\item[$\bullet$] for every $l\in[\ell]$ and $\alpha\in \Gamma_l$ an $(\alpha,\beta_l,I^\mathbf{r}_l)$-insensitive process
$\langle E_t^{l,\alpha}:t\in A^n\rangle$,
\end{enumerate}
such that for every $t\in A^n$ we have
\begin{equation} \label{e8.3}
S_t=\bigcap_{l=1}^\ell\bigcap_{\alpha\in\Gamma_l} E_t^{l,\alpha}.
\end{equation}

More generally, if\, $V$ is an $n$-dimensional combinatorial space of $A^{<\nn}$, then we say that a process
$\langle S_t:t\in V\rangle$ in a probability space $(\Omega,\mathcal{F},\mathbb{P})$ is
\emph{$(\ell,\mathbf{r})$-simplicial in~$V$} if the process $\langle S_{\mathrm{I}_V(t)}:t\in A^n\rangle$ is $(\ell,\mathbf{r})$-simplicial.
\end{defn}
\begin{rem} \label{r9.4}
In order to see the relevance of simplicial processes in this context note that if $A, n, \ell$ and $\mathbf{r}$ are as in Definition
\ref{d8.3} and $\langle S_t:t\in A^n\rangle$ is an arbitrary $(\ell,\mathbf{r})$-simplicial process, then there
exist nonempty $G\subseteq A^n$ and $x\in A^n\setminus G$ such that the set $G\cup\{x\}$ is $\ell$-separated and, moreover,
\[ \bigcap_{t\in G\cup\{x\}} S_t= \bigcap_{t\in G} S_t. \]
In particular, the events $\langle S_t:t\in A^n\rangle$ cannot be independent.
\end{rem}

\subsection{The main result} \label{ss8.2}

The following theorem, which is the main result in this section, complements Theorems
\ref{thm:dich_lines} and \ref{t7.2} and completes the analysis of correlations of stationary processes.
(We recall that the notion of stationarity for arbitrary correlations is given in Definition \ref{stationarity6}.)
\begin{thm} \label{t8.5}
Let $k,\kappa,m$ be positive integers such that $k,\kappa \meg 2$ and $\kappa\mik k^m$, and let $\ee,\sigma,\eta>0$ with
\begin{equation} \label{e8.4}
\ee \mik 1-\frac{1}{2\kappa}, \ \  \sigma \mik \frac{\ee^{\kappa- 1}}{2\kappa} \ \text{ and } \
\eta\mik \frac{\sigma}{4^{\kappa- 1}}.\end{equation}
Also let $A$ be a set with $|A|=k$, let $n>m$ be an integer, and let $\langle D_t: t\in A^n\rangle$
be an $\eta$-stationary process in a probability space $(\Omega, \mathcal{F}, \mathbb{P})$
such that $|\mathbb{P}(D_t)-\ee|\mik\eta $ for every $t\in A^n$. Then, either
\begin{enumerate}
\item [(i)] for every nonempty $G\subseteq A^n$ with cardinality at most $\kappa$ and whose type $\tau(G)$ has
dimension\footnote{Note that, by Fact \ref{f6.3}, if the type $\tau(G)$ of a  nonempty set $G\subseteq A^n$
has dimension at most $m$, then $G$ is $\ell$-separated for some $\ell\mik m$.} at most $m$ we have
\begin{equation} \label{e8.5}
\Big|\mathbb{P}\Big(\bigcap_{t \in G} D_t\Big)-\ee^{|G|}\Big|\mik \sigma,
\end{equation}
\item[(ii)] or $\langle D_t: t\in A^n\rangle$ correlates with a simplicial process when restricted to a large subspace;
precisely, there exist $\ell\in [m]$ with the following property. If\, $\mathbf{r}=(r_1,\dots,r_\ell)$
is an $\ell$-tuple of positive integers with $r\coloneqq\sum_{l=1}^{\ell}r_l\mik n-m$, then there exist an $r$-dimensional
combinatorial subspace $V$ of $A^n$ and a process $\langle S_t: t \in V\rangle$ in $(\Omega,\mathcal{F},\mathbb{P})$
which is $(\ell,\mathbf{r})$-simplicial in $V$ such that for every $t\in V$ we have
\begin{equation} \label{e8.6}
\mathbb{P}(S_t)\meg\frac{\ee^{\kappa- 1}}{4\kappa} \ \text{ and } \ \,
\mathbb{P}(D_t\, |\, S_t) \meg \ee + \frac{\sigma}{4^{\kappa-1}}.
\end{equation}
\end{enumerate}
\end{thm}
We have already pointed out that the proof of Theorem \ref{t8.5} is conceptually similar to the proofs of Theorems \ref{thm:dich_lines}
and \ref{t7.2}. More precisely, it relies on the following version of Propositions \ref{thm:step_cor_lines_simple} and \ref{p7.3}\, which,
in turn, is based on the higher-dimensional extensions of the archetypical ``projection" $t^{\beta\to\alpha}$. These
extensions are presented in Subsection \ref{ss8.3}. (See Definition \ref{d5.10} for the notions of pseudorandomness,
supercorrelation and subcorrelation which appear below.)
\begin{prop} \label{p8.6}
Let $A$ be a finite set with $|A|\meg2$, let $n, p, \ell$ be positive integers with $p+1\mik |A|^n$, let $0<\eta,\ee\mik 1$,
and let $\langle D_t:t \in A^n\rangle$ be an $\eta$-stationary process in a probability space $(\Omega,\mathcal{F},\mathbb{P})$
such that $|\mathbb{P}(D_t)-\ee|\mik\eta$ for every $t\in A^n$. Let $\mathbf{t}=(t_1,\dots, t_{p+1})$ be a tuple consisting of
distinct elements of $A^n$ with $\mathrm{s}(\mathbf{t})=\ell$, set $G\coloneqq \{t_1,\dots,t_{p+1}\}$ and $H\coloneqq\{t_1,\dots,t_p\}$,
and let $d$ denote the dimension of~$\tau(G)$. Finally, let $0<\theta,\sigma\mik 1$, let $\mathbf{r}=(r_1,\dots, r_\ell)$ be an $\ell$-tuple
of positive integers such that $r \coloneqq\sum_{l=1}^\ell r_l \mik n - d $, and assume that the process $\langle D_t:t \in A^n\rangle$
is $(\tau(H),\theta)\text{-pseudorandom}$. Then there exist an $r$-dimensional combinatorial subspace $V$ of $A^n$ and
a process $\langle S_t: t \in V\rangle$ in $(\Omega,\mathcal{F},\mathbb{P})$ which is $(\ell,\mathbf{r})$-simplicial in $V$
with the following properties.
\begin{enumerate}
\item[(i)] For every $t\in V$ we have $|\mathbb{P}(S_t)-\ee^p|\mik \theta$.
\item[(ii)] If\, $\langle D_t:t \in A^n\rangle$ is $(\tau(G),\sigma)$-supercorrelated, then for every $t\in V$ we have
\begin{equation} \label{e8.7}
\mathbb{P}(D_t\, |\, S_t)\meg\ee\, \Big(1+\frac{\sigma\ee^{-1}-\theta}{\ee^p+\theta}\Big).
\end{equation}
\item [(iii)] If\, $\langle D_t:t \in A^n\rangle$ is $(\tau(G),\sigma)$-subcorrelated, then for every $t\in V$ we have
\begin{equation} \label{e8.8}
\mathbb{P}(D_t\, |\, S_t )\mik\ee\, \Big(1-\frac{\sigma\ee^{-1}-\theta}{\ee^p-\theta}\Big).
\end{equation}
\end{enumerate}
\end{prop}

\subsection{Definitions/Notation} \label{ss8.3}

This subsection is the analogue of Subsection \ref{ss7.3}.
More precisely, let $A, n,p$ and $\ell$ be as in Proposition \ref{p8.6}. Let
\begin{equation} \label{e8.9}
\mathbf{t}=(t_1,\dots,t_{p+1})
\end{equation}
be an $\ell$-separated tuple consisting of distinct elements of $A^n$, and let $d$ be the dimension of
$\tau\coloneqq\tau(\mathbf{t})$. Also let $\mathbf{r}=(r_1,\dots,r_\ell)$ be a tuple of positive integers such that
\[ r\coloneqq\sum_{l=1}^\ell r_l\mik n-d. \]
We will define
\begin{enumerate}
\item [$\bullet$] a set $J\subseteq [d]$ with $|J|=\ell$,
\item [$\bullet$] $\beta_1,\dots, \beta_\ell\in A$,
\item [$\bullet$] an $r$-dimensional combinatorial subspace $V$ of $A^n$, and
\item [$\bullet$] for every $j\in [p]$ a map $T_j\colon V\to A^n$.
\end{enumerate}
These data are the combinatorial core of Theorem \ref{t8.5} and Proposition \ref{p8.6}.

\subsubsection{Defining $J$ and $\beta_1,\dots,\beta_\ell$} \label{sss8.3.1}

We write the type $\tau(\mathbf{t})$ as $(s_1,\dots,s_{p+1})$ where we have $s_j=\big(s_j(1),\dots,s_j(d)\big)\in A^d$
for every $j\in[p+1]$. Since $\mathbf{t}$ is $\ell$-separated, by Fact \ref{f6.3}, we have $\mathrm{s}(\tau)=\ell$.
Therefore, there exists $I\subseteq [d]$ with $|I|=\ell$ such that for every $j\in[p]$ there exists
$i\in I$ satisfying $s_j(i)\neq s_{p+1}(i)$; let $J$ be the lexicographically least set with this property, write
$J=\{\iota_1<\dots<\iota_\ell\}$, and set
\begin{equation} \label{e8.10}
\beta_1\coloneqq s_{p+1}(\iota_{1}), \dots, \beta_{\ell}\coloneqq s_{p+1}(\iota_{\ell}).
\end{equation}
Moreover, for every $j\in[p+1]$ and every $l\in[\ell+1]$ set
\begin{equation} \label{e8.11}
y_j^l=
\begin{cases}
\big( s_j(\iota_{\,l-1}+1),\dots,s_j(\iota_{\,l})\big) & \text{if } l\in [\ell], \\
\big( s_j(\iota_{\ell}+1),\dots,s_j(d)\big) & \text{if } l=\ell+1,
\end{cases}
\end{equation}
where $i_0=0$ and with the convention that $y_j^{\ell+1}$ is the empty sequence if $\iota_\ell=d$.
Notice that $s_j=y_j^1\,^{\con}\dots^\con y_j^{\ell+1}$ for every $j\in[p+1]$.

\subsubsection{Defining $V$ and the maps $\langle T_j:j\in[p]\rangle$} \label{sss8.3.2}

Set
\begin{equation} \label{e8.12}
V\coloneqq \big\{ y_{p+1}^1\!^\con z_1\!^\con\dots^\con y_{p+1}^\ell\!^\con z_\ell\!^\con y_{p+1}^{\ell+1}:
z_1\in A^{r_1},\dots,z_\ell\in A^{r_\ell} \big\}
\end{equation}
and observe that $V$ is an $r$-dimensional combinatorial subspace of $A^n$. Finally, for every $j\in[p]$ we define
$T_j\colon V\to A^n$ by the rule
\begin{equation} \label{e8.13}
T_j\big( y_{p+1}^1\!^\con z_1\!^\con\dots^\con y_{p+1}^\ell\!^\con z_\ell\!^\con y_{p+1}^{\ell+1}\big) =
y_j^1\,^\con z_1^{\beta_1\to s_j(\iota_1)}\!^\con\dots^\con y_j^\ell\,^\con z_\ell^{\beta_\ell\to s_j(\iota_\ell)}\!^\con y_j^{\ell+1}
\end{equation}
with the convention $t^{\alpha\to\alpha}=t$ for every $t\in A^{<\nn}$ and every $\alpha\in A$.

\subsubsection{Basic properties} \label{sss8.3.3}

We close this subsection with the following analogue of Fact \ref{f7.4}. The proof is straightforward.
\begin{fact} \label{f8.7}
Let $\ell,\mathbf{r}, V$ and $\langle T_j:j\in [p]\rangle$ be as above.
\begin{enumerate}
\item[(i)] For every $t\in V$ and every $1\mik i_1<\cdots < i_q\mik p$ we have
\begin{equation} \label{e8.14}
\tau\big( (T_{i_1}(t),\dots,T_{i_q}(t),t)\big)=\tau\big( (t_{i_1},\dots,t_{i_q},t_{p+1})\big)
\end{equation}
and
\begin{equation} \label{e8.15}
\tau\big( (T_{i_1}(t),\dots,T_{i_q}(t)) \big)=\tau\big( (t_{i_1},\dots,t_{i_q})\big).
\end{equation}
\item[(ii)] Let $\langle D_t:t\in A^n\rangle$ be a stochastic process in a probability space $(\Omega,\Sigma,\mu)$.
Let $\langle S_t:t\in V\rangle$ be a process of the form $S_t=\bigcap_{j=1}^p E_t^j$ where for every $j\in[p]$
either $\langle E_t^j:t\in V\rangle=\langle D_{T_j(t)}:t\in V\rangle$ or
$\langle E_t^j:t\in V\rangle=\langle D_{T_j(t)}^\complement:t\in V\rangle$. Then the process $\langle S_t:t\in V\rangle$ is
$(\ell,\mathbf{r})$-simplicial in $V$.
\end{enumerate}
\end{fact}

\subsection{Proof of Proposition \ref{p8.6}} \label{ss8.4}

Let $V$ and $\langle T_j:j\in[p]\rangle$ be the data obtained in Subsection \ref{ss8.3} for the $\ell$-separated
tuple $\mathbf{t}$ and the tuple $\mathbf{r}$. We define $\langle S_t:t\in V\rangle$ by setting
$S_t=\bigcap_{j=1}^p D_{T_j(t)}$ for every $t\in V$. By part (ii) of Fact \ref{f8.7}, we see that the process
$\langle S_t:t\in V\rangle$ is $(\ell,\mathbf{r})$-simplicial. Moreover, by \eqref{e8.15} and our assumption that
$\langle D_t:t\in A^n\rangle$ is $(\tau(H), \theta)$-pseudorandom, we have
\begin{equation} \label{e8.16}
|\mathbb{P}(S_t)-\ee^p|\mik\theta
\end{equation}
for every $t\in V$; that is, part (i) of the theorem holds true. The rest of the proof is identical to that
of Proposition \ref{p7.3}.

\subsection{Proof of Theorem \ref{t8.5}} \label{ss8.5}

Follows arguing precisely as in the proof of Theorem~\ref{t7.2} using Proposition \ref{p8.6} and the material in
Subsection \ref{ss8.3}.

%-------------------------------------------------------------------%
%                           Bibliography                            %
%-------------------------------------------------------------------%

\end{document}